\documentclass{svjour3}                     
\smartqed  

\usepackage{amsmath,amssymb}
\usepackage{graphicx}

\begin{document}
\title{Trajectory Based Models, Arbitrage and Continuity}
\thanks{Partial Support from NSERC is gratefully acknowledged}
\author{A. Alvarez and S. E. Ferrando }

\institute{Department of Mathematics, Ryerson University\\
\email{alexander.alvarez@ryerson.ca}\\
\email{ferrando@ryerson.ca}\\
}

\date{Received: date / Accepted: date}
\maketitle

\begin{abstract}

The paper develops no arbitrage results for trajectory based models by connecting these results with the usual notion of arbitrage in stochastic models. The main condition imposed, in order to avoid arbitrage opportunities,
is a local continuity requirement on the final portfolio value considered as a functional on the trajectory space. The paper shows this to be a natural requirement
by proving that a large class of practical trading strategies, defined by means of trajectory based stopping times, give rise to locally continuous functionals.
The theory is applied, with some detail, to two specific trajectory models of practical interest. The connection between trajectory based models and stochastic models is used to derive no arbitrage results for stochastic models which are not semimartingales.
\end{abstract}


\section{Introduction} \label{sec:introduction}

There have been a few attempts to propose non-probabilistic approaches to financial  market models. As examples  we mention  \cite{bick} and \cite{riedel} among others. Perfect replication, when possible, is clearly a pathwise notion, and a simple view of arbitrage is that there is a portfolio that will no produce any loss for all possible paths and there exists at least one path that will provide a profit. This informal reasoning suggests that there are areas in financial mathematics that do not
necessarily require the use of probabilities in order to obtain some meaningful results.
In fact, as shown in \cite{AFO2011} and \cite{degano}, certain aspects of financial mathematics can be treated without recourse to probabilities. This last reference proves that, in a discrete time martingale setting, the risk neutral price of an attainable option equals the min-max trajectory based price. In turn, \cite{AFO2011} establishes trajectory based perfect hedging results, and associated pricing results,  in continuous time for attainable options.

To gain perspective on the matter, we refer to \cite{follmer2} for a discussion of the implications of probabilistic assumptions in finance and, in particular, the section on Knightian uncertainty. 

Results obtained without recourse to any probabilistic assumptions are more generally applicable and as such are of interest to a conscientious practitioner who tries to check the reality of the myriad of assumptions needed when using a particular result from financial mathematics. In such a way, our results are robust in the sense that they are independent of the particular probabilistic model that the stock may follow. In this respect, our work is conceptually close to some recent literature on robust modelling, we mention \cite{nutz} and \cite{vorbrink} as representatives of this literature.

In a stochastic setting, some aspects of modelling may naturally fall under Knightian uncertainty. An example is provided by  the notion of crash in portfolio optimization (\cite{desmettre}) where, the number, timing and size of a downwards stock change (a {\it crash}) is treated without probabilistic assumptions. 


In the present paper, instead of starting with a probability space $(\Omega, \mathcal{F}, (\mathcal{F}_t)_{t \geq 0}, P)$ and modelling the stock as a stochastic process $X$, we propose to concentrate
on a trajectory space $\mathcal{J} \subseteq \mathcal{D}[0,T]$ where the latter
is the set of functions $x:[0,T]\rightarrow \mathbb{R}$ which are right continuous with left limits (RCLL). The classical paradigm is to look for
a probability $P$ to model the unfolding market, the proposed approach focuses
on the set $\mathcal{J}$ which is conveniently treated as a metric space 
$(\mathcal{J}, d)$. A main question addressed in the present paper is: can we obtain general conditions that are 
practically relevant and  that guarantee trajectory markets to be arbitrage free?

Our general technique and framework are the ones 
introduced in  \cite{AFO2011} where several non-probabilistic (NP) no-arbitrage and hedging results were obtained. The main technique, encapsulated in Theorem \ref{main-arbitrage-local-cont} of the present paper, shows that one can go back and forth between stochastic
arbitrage in $(\Omega, \mathcal{F}, (\mathcal{F}_t)_{t \geq 0}, P)$ and NP arbitrage in 
$\mathcal{J}$. Theorem \ref{main-arbitrage-local-cont} is a simple result in the sense that the complexity of the problem tackled is hidden under the following two hypothesis required to apply the theorem:
\begin{enumerate}
\item Small Ball Property: the sets of  
paths in  $\{X(\omega): \omega \in \Omega\}$ which are
arbitrarily close (in the metric $d$) to arbitrary elements $x \in \mathcal{J}$ are non negligible under $P$.

\item Local $V$-Continuity of Portfolios: the terminal portfolio value is a locally continuous function (see Definition \ref{local-vContinuity}) as a function on 
$\mathcal{J}$ with respect to metric $d$.
\end{enumerate}

These two conditions, in a slightly different form, were already used in  \cite{valkeila-2} to prove no arbitrage results in models that are related to the Black-Scholes model but are not necessarily semimartingales. In that paper, the authors
work exclusively with the metric induced by the uniform norm. We realized that these two conditions together, but now within the framework provided
by a general metric structure in $\mathcal{J}$, can be useful for the study of 
more realistic models (see Sections 5 and 6). 

From a purely financial point of view the metric $d$ is not required in the sense that most concepts 
are defined independently of $d$ (for example, the concept of NP arbitrage). This means that the metric $d$ can be conveniently chosen over $\mathcal{J}$ so that the two conditions above are satisfied and Theorem \ref{main-arbitrage-local-cont} can be applied. This flexibility to choose an appropriate metric $d$, makes our approach quite flexible
and therefore powerful to deal with some models. An example of 
this is the novel metric $d_{QV}$
used in Section 6. 


Distinctive characteristics of trajectory based models include:
\begin{itemize}
\item Chart trajectories are directly observable.

\item It generalizes the modelling with stochastic process where
a trajectory set is implicit (the support of the process).

\item The generality of the framework allows to obtain results for non-semimartingale models.
\end{itemize}
Fundamental results by Delbaen and Schachermayer on non-semimartingale models
  imply the existence of a free lunch with vanishing risk in the class of simple portfolios
  (see, for example, \cite{delbaen}). Therefore, the justification and use of non-semimartingale
process in financial modelling is a delicate matter; what many researchers have done to deal with non-semimartingale
models is to restrict the class of allowed portfolios (see among others \cite{cheridito}, \cite{valkeila-2}, \cite{jarrow} and \cite{bender2}). 
Our Theorem \ref{main-arbitrage-local-cont}, item $ii)$, provides a tool to establish no-arbitrage results in non-semimartingale models. More precisely, by restricting to locally $V$-continuous portfolios, we conclude that there is no arbitrage in 
$(\Omega, \mathcal{F}, (\mathcal{F}_t)_{t \geq 0}, P)$ as long as the corresponding trajectory space $\mathcal{J}$ admits no NP arbitrage. For this conclusion, the fact that $X$ is a semimartingale or not is irrelevant. We provide
examples of such applications in Sections
\ref{implicationsForStochasticFrameworks} and \ref{modifiedHestonModel}. 
The class of portfolios covered by our results includes simple portfolios (see Theorem \ref{simplePortfoliosAreLC}) and also portfolios that are continuously rebalanced
between stopping times (see Theorems \ref{continuousRebalancingPortfolios} and 
\ref{loc-cont-portfolio-under-QV}); it then follows that local $V$-continuity under some
metric is a natural restriction on portfolios to avoid arbitrage in many non-semimartingale models. In other words, a consequence of our results is that for some models that are not semimartingales, and for which an arbitrage strategy exists, then it must necessarily be non locally V-continuous.  Proving the local $V$-continuity of the above mentioned classes of portfolios represents the bulk of our technical work.


As already mentioned, we build on the framework
of reference \cite{AFO2011} to which we will refer to avoid any unnecessary duplication. A main contribution of the present paper is to incorporate
stopping times in the formalism of \cite{AFO2011}. Towards this goal, we rely on a notion of trajectory based stopping times that is implicit in the usual stopping times (i.e. the formulation based on filtrations) and both notions are related
(\cite{shiryaev}); differences between the two concepts are highlighted in \cite{boshuizen} and \cite{hill}. The introduction of stopping times allows  to substantially enlarge the class of portfolios for which we can prove no-arbitrage results. Handling stopping times is a technically challenging problem as infinite sequences of stopping
times have to be proven to be jointly strong locally continuous (as per Definition \ref{joint-strong-LC}). These results are notably different and more difficult to obtain than those in \cite{AFO2011}, as the main examples treated in that paper satisfied the $V$-continuity property instead of local $V$-continuity. Another main contribution that sets us apart from \cite{AFO2011} is a first analysis of a variable volatility trajectory class treated in detail in Section \ref{variableVolatilityClass}
and related stochastic volatility models. Nothing in  \cite{AFO2011}, or in any other
reference that we are aware of, covers this type of non-semimartingale stochastic volatility models.

 A key problem faced by a trajectory based approach is to be able to integrate with respect to functions of unbounded variation given that portfolio values
are represented by such integrals. In this paper, we  present 
the general NP framework in Section 2 without mentioning any specific type of integral.  The main reason for this is that some results (for example Theorem \ref{simplePortfoliosAreLC}) can be proven without mentioning a specific integral
as it only involves simple portfolios. Later on, in particular in Sections 5 and 6, we specifically use Follmer's integrals, see \cite{follmer}. Given the recent surge of works (see for example \cite{cont2}, \cite{guasoni} and \cite{perkowsky}) that study pathwise integrals and some related properties, we can see a  potential for the application of our general framework using different integrals depending on the trajectory space.

In the context of the present paper a main advantage of the proposed point of view is the ability to obtain no-arbitrage results for non-semimartingale models. This is achieved through a clear methodology
using the small ball property and local $V$-continuity mentioned above. The proposed approach is flexible and general, it can be deployed with different metrics as well as  different integrals. At this point, 
another question arises naturally: is it possible to obtain these no-arbitrage results in non-semimartingale models without explicitly using the NP framework  and along the lines of the approach in  \cite{valkeila-2}, for example?  We do not have a definitive response to that question, but we do believe that without the NP formalism, it will not be very intuitive to introduce a {\it convenient} metric structure and make it a central element in the analysis of ``small balls" properties and local continuity. As evidence that this type of developments under a general metric are not natural from within the classical framework, we should mention that in many papers that make use of the 1 ``full support" property, 
for example \cite{guasoni2} and \cite{pak}, the metric structure induced by the uniform norm is used implicitly  and we have never found any hint suggesting the use of a different metric. Introducing the trajectory based framework makes the choice of a metric and subsequent analysis a lot more intuitive and natural.

The paper is organized as follows. Section \ref{nPFramework} introduces our main definitions, in particular we provide the definition of  NP-market and trajectory based stopping time and draw some basic
consequences from this last concept.
 Section \ref{locallyContinuousPortfolios} introduces a notion of local continuity for a general metric space; under general assumptions, simple portfolios defined through a sequence of trajectory based stopping times are shown to define portfolios with an associated locally continuous value functional. Section \ref{arbitrage}, following \cite{AFO2011}, provides a result linking the usual (probabilistic) notion of arbitrage with
NP-arbitrage. This connection is achieved by assumptions of local continuity and  small balls allowing to transfer results,  back and forth, between NP-market models and  stochastic market models.
Section \ref{nPJumpFiffusionClass} constructs a NP-jump-diffusion arbitrage free market and shows how the result can be used to prove that several non-semimartingale market models are arbitrage free. Section \ref{variableVolatilityClass} constructs a trajectory dependent volatility arbitrage free market and also draws implications to related non-semimartingale market models.
Finally, Section \ref{overview} provides an overall perspective on the trajectory based approach and concludes.
The Appendix contains statements and proofs of technical results used in the paper.

\section{Non Probabilistic Framework} \label{nPFramework}

\subsection{Non Probabilistic Market}

Most of the definitions and ideas in this brief section were already
introduced in \cite{AFO2011}, we include them here to make the paper as self-contained as possible.

\noindent
Let $x$ be a real valued function on $[0, T]$ which is right continuous and has left limits (RCLL for short), the space of such functions will be denoted by $\mathcal{D}[0,T]$. We assume the existence of a non risky asset evolving with constant interest
rate $r$ which, for simplicity, will be set to  $r=0$.
The risky asset is modeled  by trajectories of functions $x$ belonging to certain class $\mathcal{J}(x_0)$, which we will write simply as $\mathcal{J}$,
where $x(0) = x_0$ for all $x \in \mathcal{J}$.

We also assume that, for every trajectory $x \in \mathcal{J}(x_0) \subset \mathcal{D}[0,T]$, the integrals
\begin{equation}\label{integrability-assumption}
\int_0^t y(s,x) dx_s
\end{equation}
 are well defined for all $t \in [0,T]$ under appropriate conditions on the integrand $y$. The sense in which these integrals exist is not specified yet but a general property will be assumed at this point. In the case that $y(\cdot, x)$ is piecewise constant, namely for any $t \in [0, T]$:
\begin{equation}  \nonumber
y(t,x)=1_{[0,t_{1}]}(t)~c_0(0,x)+\sum_{i=1}^{n(x)-1} 1_{(t_i,t_{i+1}]}(t)~c_i(t_i,x),
\end{equation}
where  $0=t_0< t_1(x)< ...< t_{n(x)}(x)=T$ is a finite, $x$-dependent partition, we will require that for all $t \in [0, T]$
\begin{equation} \label{minimumIntegralRequirement}
\int_0^t y(s,x) dx_s = \sum_{i=0}^{k(x)- 2} c_i(t_i,x) \left[ x_{t_{i+1}}-x_{t_i}\right] + 
c_{k-1}(t_{k-1},x) \left[ x_{t}-x_{t_{k-1}}\right],
\end{equation}
where $k(x)$ is the smallest integer such that $ t \leq t_{k(x)}$. 

There are several notions of integrals suitable for pathwise integration (\cite{follmer}, \cite{guasoni}).
In principle, different trajectory classes $\mathcal{J}$ may require each a specific notion of integration
and we will indicate the use of each such integrals whenever appropriate.

\vspace{.1in}
A NP-portfolio $\Phi$ is a function $\Phi$$:$ $[0, T] \times
\mathcal{J}(x_0) \rightarrow \mathbb{R}^2$, $\Phi =(\psi, \phi)$,
satisfying $\Phi(0,x)=\Phi(0,x')$ for all $x, x' \in \mathcal{J}(x_0)$. This common value
will be denoted $\Phi(0, x_0)$.
We will also consider the associated projections $\Phi_x$$:$$ [0,T]
\rightarrow \mathbb{R}^2$ and $\Phi_t$$:$$ \mathcal{J}(x_0) \rightarrow
\mathbb{R}^2$, for fixed $x$ and $t$ respectively.

The value of a NP-portfolio $\Phi$ is the function $V_{\Phi}$$:$$
[0,T]\times \mathcal{J}(x_0) \rightarrow \mathbb{R}$ given by:
\begin{equation}  \nonumber 
V_{\Phi}(t,x) \equiv \psi(t, x)+ \phi(t,x) ~x(t).
\end{equation}

\begin{definition} \label{def:NP-conditions}
Consider a class $\mathcal{J}(x_0)$ of trajectories starting at $x_0$ and consider a NP-portfolio $\Phi$:

\begin{itemize}
\item [i)] $\Phi$ is said to be NP-predictable
if $\Phi_t(x) = \Phi_t(x')$ for all $x, x' \in \mathcal{J}(x_0)$
such that $x(s) = x'(s)$ for all $0 \leq s < t$ and $\Phi_x(\cdot)$
is left continuous and has right limits (LCRL for short) for all $x \in \mathcal{J}(x_0)$.
\item [ii)] $\Phi$ is said to be NP-self-financing if the integrals $\int_0^t \psi(s,x)~ ds$
and $\int_0^t \phi(s,x) d x_s$ exist for all $x \in
\mathcal{J}(x_0)$ in the senses of Stieltjes and expression (\ref{integrability-assumption}) respectively
and
\begin{equation} \label{selfFinancing0}
V_{\Phi}(t,x)=V_0 + \int_0^t  \psi(s, x) ~r~ds + \int_0^t \phi(s, x) d x_s, \; \forall x \in \mathcal{J}(x_0),
\end{equation}
where $V_0 = V_{\Phi}(0,x)=\psi(0,x) + \phi(0,x) ~x(0)$ for any $x \in \mathcal{J}(x_0)$.
\end{itemize}
\end{definition}

\begin{remark} \label{selfFinancingBankInvestment}
Consider $r=0$ and function $\phi(\cdot, \cdot)$ given; define $\Phi = (\psi , \phi)$ where    $\psi(t, x) \equiv  V_{\Phi}(t^-, x) - x(t^-) \phi(t,x)$
and $V_{\Phi}(t^-, x)$  is given by (\ref{selfFinancing0}) with $r=0$.
For all the family of functions $\phi$ considered in this paper, and under the working assumption $r=0$, these portfolios $\Phi$ will satisfy all the properties listed in Definition \ref{def:NP-conditions}.
We will not prove this fact in each instance but refer to
Theorem \ref{simplePortfoliosAreLC} as a typical example.
\end{remark}

\begin{definition}\label{def:NP-market}
A NP-market model $\mathcal{M}$ is a pair $\mathcal{M}= (\mathcal{J},\mathcal{A})$ where $\mathcal{J}$ represents a class of
possible trajectories for a risky asset and
 $\mathcal{A}$ is a class of NP-portfolios.
\end{definition}

For some of our results, we will need to require the following stronger hypothesis of admissibility.
\begin{definition}  \label{npAdmissiblePortfolio}
A NP-portfolio $\Phi$ is said to be NP-admissible if  $V_{\Phi}(t,x) \geq - A $, for a constant $A = A(\Phi) \geq 0$, for all $t \in [0, T]$ and all $x \in \mathcal{J}(x_0)$.
\end{definition}

The following definition provides the notion of arbitrage in a non probabilistic framework.

\begin{definition}\label{nPArbitrage}
A NP-portfolio $\Phi$ defined on a trajectory space $\mathcal{J}$ is a NP-arbitrage if:
\begin{itemize}
\item $V_0 =0 $ and   $V_{\Phi}(T,x) \geq 0$, $\forall x \in \mathcal{J}$.
\item $\exists x^{\ast} \in \mathcal{J}$ satisfying $V_{\Phi}(T, x^{\ast}) > V_{\Phi}(0, x^{\ast}) $.
\end{itemize}
We will say that the NP-market $\mathcal{M}= (\mathcal{J},\mathcal{A})$
is arbitrage free if $\Phi$ is not a NP-arbitrage, for each  $\Phi \in \mathcal{A}$.
\end{definition}

\subsection{Trajectory Based Stopping Times}

The usual stopping times depend on a given filtration but a closely related notion can be defined in a trajectory based sense. 
\begin{definition} \label{definitionOfNPStoppingTimes}
Let $\mathcal{J}$ be a class of trajectories, a functional $\tau: \mathcal{J} \rightarrow [0,T]$ is called a NP-stopping time on
$\mathcal{J}$ if for every pair of trajectories $x,y \in \mathcal{J}$, with $x(s) = y(s)$ for all $s \in [0, \tau(x)]$, it follows that
$\tau(y) = \tau(x)$.
\end{definition}
\begin{remark} \label{spaceDiscretizationOfStoppingTimes}
Unless indicated otherwise, when proving a certain functional to be a trajectory based stopping time, its domain $\mathcal{J}$ will be taken to be the whole set of RCLL functions ($\mathcal{D}([0,T])$.) This
approach provides more general results because,
once the result is obtained on $\mathcal{D}([0,T])$, it applies to any arbitrary subset.
\end{remark}

\vspace{.1in}
\noindent
It will take a separate study to derive systematically the consequences following from Definition \ref{definitionOfNPStoppingTimes}, we content ourselves with
providing some basic results, some of them will be used in the remaining of the paper. We also refer to \cite{boshuizen} and \cite{hill} for closely related developments.

It is well known that the usual stopping times (i.e. the a filtration based formulation) can be equivalently recast in  terms of trajectories.
See Theorem 7, Chapter 1,  in \cite{shiryaev} (this is sometimes referred to as Galmarino's test). As an illustration, we formulate here a particular version of this type of result.
Define, for each $s\in [0,T]$, the functions $X_s:\mathcal{J} \rightarrow \mathbb{R}$
by $X_s(x) = x(s)$ and the associated canonical filtration on a given trajectory space ${\mathcal{J}}$:
$$
\mathcal{F}_t^{\mathcal{J}}= \sigma(X_s: 0 \leq s \leq t),
$$
where we have considered $\mathbb{R}$ with the Borel sigma algebra.
We then have the following result (see Problem 2.2 from \cite{karatzas}):
\begin{proposition} \label{connectionToRandomTimes}
A stopping time $\tau$ relative to the filtration $\{\mathcal{F}_t^{\mathcal{J}}\}$ is also a trajectory based stopping time.
\end{proposition}
Proposition \ref{connectionToRandomTimes} shows that  stopping times with respect to the canonical filtration associated to the process $X=\{X_s\}_{0 \leq s \leq T}$
will also be trajectory based stopping times.

\vspace{.2in}
Proposition \ref{stopTimesInCanonicalFiltrationAreNPStpTimes} provides several examples of stopping times.
The notation below makes use of the convention $\inf_{t \in [0,T]} \varnothing \equiv T$.
We use the following notation throughout the paper: $f(t^{-})$ represents the left limit at $t$ for the function $f$.

\begin{proposition} \label{stopTimesInCanonicalFiltrationAreNPStpTimes}
The following are trajectory based stopping times on $\mathcal{D}([0,T])$:
\begin{enumerate}
\item  $\tau(x)=c$ if $c \in [0,T]$. \label{constant}
\item  $\tau(x)= (\tau_1(x) + \tau_2(x)) \wedge T$ \label{sum}
where $\tau_1, \tau_2$ are trajectory based stopping times on $\mathcal{D}([0,T])$.
\item Let $A \subset \mathbb{R}$ be a closed set. \label{hittingTime}
For all $x \in \mathcal{J}$, define $\tau(x)= \inf\{t \in [0,T]: x_t \in A\}$. Then
$\tau$ is a trajectory based stopping time.
\item  $\tau(x)=\inf_t \left\{ x(t) \geq a \right\}$. \label{crossingLevel}
\item Consider $\delta > 0$, the \label{deltaJumps}
following functional is a trajectory based stopping time:
$$
\tau_{\delta}(x) = \inf_{t \in [0, T]} \{|x(t) - x(t^-)| > \delta\} .
$$

\item The minimum of a finite collection of trajectory based stopping times is also a  trajectory based stopping time. \label{minimum}
\item Let $\tau_t$ be a collection of  trajectory based stopping times indexed by $t \in I$, \label{supremum}
where $I$ is an arbitrary index set. Then $\tau(x) = \sup_{t \in I}~\tau_t(x)$ is a trajectory based stopping time.
\end{enumerate}
\end{proposition}
\begin{proof}
Items \ref{constant} and \ref{sum} have immediate proofs.
To prove \ref{hittingTime}, consider $x, y \in \mathcal{D}([0,T])$ such that $x(s) =y (s)$ for all $s \in [0, \tau(x)]$. In the case that   $x(\tau(x)) \in A$ we have that
 $y(\tau(x)) \in A $ and $y(s) \notin A$ for all $s \in [0, \tau(x))$. Hence $\tau(y) = \tau(x)$. The case $x(\tau(x)) \notin A$ is not possible as we argue next; for each $\epsilon_n >0$ there exists
$t_{n}$ satisfying $\tau(x) \leq t_{n} < \tau(x) + \epsilon_n $ with $x(t_n)  \in A$. Then, this implies, by right continuity and the fact that $A$ is closed, that
$\lim_{t_n \searrow \tau(x)} x(t_n) = x(\tau(x)) \in A$.

\noindent
The result  \ref{crossingLevel} follows from \ref{hittingTime} by taking $A = [a, \infty)$.

\noindent
To prove \ref{deltaJumps}  assume
 $x,y$ to be RCLL functions satisfying $x(s) = y(s) $ for all $s \in [0, \tau_{\delta}(x)]$. We analyze two cases: Case $i)$ when $|x(\tau_{\delta}(x)) - x(\tau_{\delta}(x)^-)| > \delta$, it follows that
$|y(\tau_{\delta}(x))- y(\tau_{\delta}(x)^-)| > \delta$ as well. Moreover, $|y(t) - y(t^-)| > \delta$ for $t \in [0, \tau(x))$ is impossible as it will contradict the definition of $\tau_{\delta}(x)$
as an infimum. It follows then, that for this case, $\tau_{\delta}(x)= \tau_{\delta}(y)$. We will now argue that case $ii)$, namely, $|x(\tau_{\delta}(x)) - x(\tau_{\delta}(x)^-)| \leq  \delta$ does not occur, this will conclude
the proof. Consider $\epsilon_n  \searrow 0$, then there exists $ \tau_{\delta}(x) < t_n \leq \tau_{\delta}(x)+ \epsilon_n$ such that $|(x(t_n) - x(t_n^-)| > \delta$. These last statements contradict the fact that  
$\lim_{n \rightarrow \infty} (x(t_n) - x(t_n^-))=0$  which follows from the fact that $x \in \mathcal{D}([0,T])$,
$ \tau_{\delta}(x) \in [0, T)$ and $t_n \searrow \tau_{\delta}(x)$.

\noindent
Items \ref{minimum} and \ref{supremum}, have direct proofs.
\qed
\end{proof}

\begin{corollary}
Consider $\mathcal{J}$ to be a fixed subset of  $\mathcal{D}([0,T])$ such that  each $x \in \mathcal{J}$ has a finite number of jumps.
Then,
$$
\tau(x) = \inf_{t \in [0, T]} \{(x(t) - x(t^-)) \neq 0\}= \inf_{\delta >0} \inf_{t \in [0, T]} \{ |x(t) - x(t^-)| > \delta\},
$$
is a trajectory based  stopping time on $\mathcal{J}$.
\end{corollary}
\begin{proof}
Clearly, from the hypothesis on finite number of jumps,  $\inf_{t \in [0, T]} \{(x(t) - x(t^-)) \neq 0\} \geq \inf_{\delta >0} \inf_{t \in [0, T]} |x(t) - x(t^-)| > \delta\}$.
The reverse inequality follows by noticing that $\inf_{t \in [0, T]} \{|x(t) - x(t^-)| > 0\} \leq  \inf_{t \in [0, T]} |x(t) - x(t^-)| > \delta\}$ for all $\delta >0$.
So, $\tau(x) = \inf_{\delta >0} \tau_{\delta}(x)$ and each $\tau_{\delta}$ is a trajectory based stopping time according to  Proposition \ref{stopTimesInCanonicalFiltrationAreNPStpTimes}, item \ref{deltaJumps}.
Again, from the hypothesis on finite number of jumps, for a fixed $x$, there exists $\beta= \delta(x)$ such that $\tau(x) = \tau_{\beta}(x)$, therefore, for such fixed $x$, consider $y$ such that
$x(s) =y(s)$ for all $s \in [0, \tau(x)]$. It follows that $\tau_{\beta}(x)= \tau_{\beta}(y)$ and so $\tau(x)= \tau(y)$.
\qed
\end{proof}




\section{Locally Continuous Portfolios}  \label{locallyContinuousPortfolios}

In \cite{valkeila-2} the concept of local continuity is introduced as follows.

\begin{definition}
Let $\mathcal{X}$ and  $\mathcal{Y}$ be metric spaces. A function $f: \mathcal{X} \rightarrow \mathcal{Y}$
is locally continuous if for all $x \in \mathcal{X}$ there exists an open $U_x \subset \mathcal{X}$ such that
$x \in \overline{U}_x$ and $f(x_n) \to f(x)$ whenever $U_x \ni x_n \to x$.
\end{definition}
The notion of local continuity is equivalent, at least in the setting of metric spaces, to quasicontinuity, a notion developed in the literature (see \cite{kempisty}).
Local continuity is a natural topological property for NP-stopping times, this is in contrast to the stronger property of continuity.
As an example we mention that
the stopping time given in item \ref{crossingLevel} of Proposition \ref{stopTimesInCanonicalFiltrationAreNPStpTimes} is not continuous with respect to the uniform metric but it is easily seen to be locally continuous relative to that metric.

\vspace{.1in}
\noindent
In our framework, we will not only need  locally continuous NP-stopping times but we will also require they satisfy the following stronger continuity property.

\begin{definition} \label{strongLocallyContinuous}
Let $\mathcal{J}$ be a class of trajectories provided with metric $d$. A stopping time $\tau$ on $\mathcal{J}$
is said to be strong locally continuous if for all $x\in \mathcal{J}$ there exists an open set $U_x \subset \mathcal{J}$
such that $x \in \overline{U}_x$ and whenever $U_x \ni x_n \to x$:
\begin{enumerate}
\item $\tau(x_n) \to \tau(x)$ (\mbox{local continuity)}. \label{condition1}
\item $x_n(\tau(x_n)) \to x(\tau(x))$. \label{condition2}
\end{enumerate}
\end{definition}

The next proposition shows that condition \ref{condition2}, in Definition \ref{strongLocallyContinuous},
follows from local continuity in the case of the uniform metric. The proof is straightforward and hence omitted.

\begin{proposition}
Let $\mathcal{J}$ be a class of continuous trajectories provided with the topology induced by the uniform norm. If $\tau$
is a locally continuous stopping time on $\mathcal{J}$ then $\tau$ is strong locally continuous on $\mathcal{J}$.
\end{proposition}


\begin{definition} \label{local-vContinuity}
Let $(\mathcal{J},\mathcal{A})$ be a NP-market. A NP-portfolio
$\Phi \in \mathcal{A}$ is said to be locally V-continuous with respect to $d$
if the functional $V_{\Phi}(T,\cdot)$$:$$ \mathcal{J} \rightarrow \mathbb{R}$
is locally continuous with respect to the topology induced on $\mathcal{J}$ by
the distance $d$.
\end{definition}

In contrast to the use of $V$-continuity
in \cite{AFO2011}
we will consider locally $V$-continuous portfolios. The reason for this is that, in general, typical NP-stopping times  
will generate NP-portfolios that are only locally V-continuous.
The following proposition provides an example of a simple portfolio, defined by a constant NP-stopping time, that
is locally V-continuous but not V-continuous. This example was considered in Propositions 4 and 5 from \cite{AFO2011}.
We briefly recall the definition of the required trajectory class $\mathcal{J}^{a, \mu}(x_0)$
from \cite{AFO2011}.

Denote by $\mathcal{N}([0, T])$ the collection of all
functions $n(t)$ such that there exists a non negative integer $m$
and positive numbers $0 < s_1 < \ldots < s_m < T$ such that
$n(t)=\sum_{i \geq 1} 1_{[0,t]}(s_i)$. The function $n(t)$ is
considered as identically zero on $[0, T]$ whenever $m=0$.
\begin{itemize}
\item Given constants $\mu, a \in \mathbb{R}$, $\mu ~a <0$ and $x_0 > 0$,
let $ \mathcal{J}^{a,\mu}(x_0)$ to be the class of all functions
$x$ for which exists $n(t) \in \mathcal{N}([0, T])$ such that:
\begin{equation} \label{eq:poisson}
x(t)=x_{0}e^{\mu t}(1+a)^{n(t)}.
\end{equation}
\end{itemize}
The function $n(t)$ counts the number of jumps present in the path
$x$ until, and including, time $t$.

We will consider that the class of trajectories $\mathcal{J}^{a,\mu}(x_0)$
is endowed with the Skorohod's distance $d_S$. The Skorohod's distance 
between functions in $D[0,T]$ is defined as follows. Let $\Lambda$ denote the class
of strictly increasing, continuous mappings of $[0,T]$ onto itself,
then 
$\displaystyle{d_S(x,y)=\inf_{\lambda \in \Lambda} \max \{ \|\lambda-I \|, \| x-y\circ \lambda\|\}}$
where $I:[0,T] \to [0,T]$ is the identity function. More on the Skorohod's metric
can be found in \cite{billingsley}. We have the following result:

\begin{proposition}
Consider $T=1$ and the NP-portfolio with initial value
$x_0$ defined by:
\begin{itemize}
\item $\phi(t,x)=1$, $\psi(t,x)=0$, for all
$x \in \mathcal{J}^{a, \mu}(x_0)$ if $0 \leq t \leq 1/2.$
\item $\phi(t,x)=0$, $\psi(t,x)=x_{\frac{1}{2}}$ for all $x \in \mathcal{J}^{a, \mu}(x_0)$ if $1/2<t \leq 1$.
\end{itemize}
Then, $\Phi = (\psi, \phi)$ is a NP-admissible portfolio that is locally V-continuous but not continuous with respect to the Skorohod's metric on 
$\mathcal{J}^{a, \mu}(x_0)$.
\end{proposition}
\begin{proof}
It is easy to see that $\Phi$ is NP-admissible, the fact that it is not V-continuous is proven in Proposition 4 of
\cite{AFO2011}.
Let $x \in \mathcal{J}^{a, \mu}(x_0)$. If $x(1/2)-x(1/2-)=0$ (meaning that $x$ doesn't jump
at $t=1/2$) then $V_{\Phi}(T,\cdot)$ is continuous at $x$, hence locally continuous at
$x$.\\
If $x(1/2)-x(1/2-) \neq 0$ consider,
\begin{equation} \nonumber
U_x^{\epsilon}=\{ y \in \mathcal{J}^{a, \mu}(x_0): 0< d_S(y,x)< \epsilon,
 y(1/2)=x(1/2), y(1/2-) \neq x(1/2-)\}.
\end{equation}
Then: $(i)~x \in \overline{U}_x^{\epsilon}$ and  $(ii)~V_{\Phi}(t,x_n) \to V_{\Phi}(t,x)$ if $x_n \to x$ in $U_x^{\epsilon}$.
\qed
\end{proof}

A basic class of portfolios is defined through sequences of NP-stopping times; these sequences are introduced in the following definition and the associated portfolios are introduced in Definition \ref{simplePortfolios}.

\begin{definition}[(Unbounded) Finite Sequence of Stopping Times] \label{unboundedSequenceOfStoppingTimes}
Let $\mathcal{J}$ be a class of trajectories and consider a non decreasing sequence $\tau= \{\tau_n\}$ of
NP-stopping times $0=\tau_0 \leq \tau_1 \leq \tau_2\leq \cdots \leq T$ such that, for each $x \in \mathcal{J}$, there is a smallest integer $M(x)$ satisfying
$\tau_{M(x)}(x)= T$. Such a sequence  is said to be a  finite sequence of stopping times.
\end{definition}

\noindent
The case of a bounded number of stopping times $0 = \tau_0 \leq \tau_1 \leq \ldots \leq \tau_N=T$ is covered by the above definition
by taking $M(x) =N$ for all $x$.

\begin{definition}[Joint Strong Locally Continuity] \label{joint-strong-LC}
Let $\mathcal{J}$ be a class of trajectories provided with metric $d$. Consider a finite sequence of NP-stopping times   $\tau = \{\tau_n\}$ as per Definition \ref{unboundedSequenceOfStoppingTimes}. Such a sequence  is said to
be jointly strong locally continuous on $\mathcal{J}$, with respect to $d$,
if for all $x\in \mathcal{J}$ there exists an open set $U_x \subset \mathcal{J}$
such that $x \in \overline{U}_x$ and whenever $U_x \ni x_n \to x$:
\begin{itemize}
\item i)  $\displaystyle{\lim_{n \to \infty} \tau_i(x_n) = \tau_i(x)}$ for all $i$.
\item ii) $\displaystyle{\lim_{n \to \infty} x_n(\tau_i(x_n)) = x(\tau_i(x))}$ for all $i$.
\item iii)
$\displaystyle{\lim_{n \to \infty}M(x_n) = M(x)}$.
\end{itemize}
\end{definition}

Definition \ref{joint-strong-LC} indicates that all stopping times $\tau_i$ are
strong locally continuous, not only individually, but jointly in the sense that
the open subset $U_x \subset \mathcal{J}$ is common to all stopping times.

\begin{remark} \label{finiteNumberOfStoppingTimes}
In the case of a finite number of stopping times $M(x) =N$ for all $x$, and so item $iii)$ in Definition \ref{joint-strong-LC} holds for any $\mathcal{J}$ and $d$. Moreover, when $N=2$, which represents the case of a single NP-stopping time, Definition \ref{joint-strong-LC} coincides with Definition \ref{strongLocallyContinuous}
 and so, the requirement of joint strong local continuity reduces to
strong locally continuity.
\end{remark}

\begin{definition} [Simple Portfolios] \label{simplePortfolios}
Assume as given: $\tau = \{\tau_n\}$ a finite sequence of stopping times as per Definition \ref{unboundedSequenceOfStoppingTimes}, functions $\phi_0(\cdot), \phi_1(\cdot),...$ defined  on $\mathcal{J}$
satisfying $\phi_i(x) = \phi_i(\hat{x})$ for any  $x, \hat{x} \in \mathcal{J}$ satisfying  $x(s) = \hat{x}(s)$, $0 \leq s \leq \tau_i(x)$,  and a real number $V_0$.
Fix $x \in \mathcal{J}$, $t \in [0, T]$ and define:
\begin{equation} \label{portfolioFromNPStoppingTimesAtZero}
\phi(t,x) \equiv \phi_{0}(x)~ 1_{[0, \tau_{1}(x)]}(t)+ \sum_{k \geq 1}~ \phi_k(x) 1_{(\tau_k, \tau_{k+1}]}(t).
\end{equation}
For fixed $x$ and   $t  \in (\tau_{j}(x) , \tau_{j+1}(x)]$, define:
\begin{equation}\label{portfoliosViaStoppingTimes}
V(t,x)=V_0+\sum_{k=0}^{j-1}~\phi_{k}(x)\left[ x(\tau_{k+1}(x))-x(\tau_{k}(x)) \right]+
\phi_{j}(x)\left[ x(t)-x(\tau_{j}(x)) \right]
\end{equation}
 and $~V(0,x) = V_0$. Finally, define: $\psi(t,x) = V(t^{-},x)- \phi(t,x)~ x(t^{-})$ and NP-portfolio strategy $\Phi=(\psi, \phi)$. $\Phi$ will be said to be a NP-simple portfolio
associated to the sequence $\tau =\{\tau_n\}$.
\end{definition}

\begin{theorem}  \label{simplePortfoliosAreLC}
Let $\tau = \{\tau_n\}$ be as in Definition \ref{joint-strong-LC} and consider the associated simple portfolio $\Phi$
as in Definition \ref{simplePortfolios} and further assume that the functions $\phi_k$ appearing
in (\ref{portfolioFromNPStoppingTimesAtZero}) are continuous functions.
Then, $\Phi$ is a NP-portfolio  that is NP-predictable, NP-self-financing  and locally V-continuous.
\end{theorem}
\begin{proof}
Notice that $\phi(0,x) = \phi_0(0,x)$ and the dependency of $ \phi_0(0,x)$ on $x$ is only through $x(0)$,
it follows that
 $\Phi(0,x)= \Phi(0,x')$ for any $x, x' \in \mathcal{J}$. To prove NP-predictability of $\Phi$ consider $t \in (0, T]$ and
 $x(s) = y(s)$,~ $0 \leq s <t$, ~ $x, y \in \mathcal{J}$. Let $n$ be the largest integer such that $\tau_n(x) < t$, such an integer exists
 because $\tau_0(x)=0$, $\tau_{M(x)}(x) = T$ and $\tau_i \leq \tau_{i+1}$. It follows that $\tau_n(x) = \tau_n(y)$, $\phi(t,x) = \phi_n(x)$
 and $\tau_{n+1}(x) \geq t$. Also $\tau_{n+1}(y) \geq t$ otherwise $\tau_{n+1}(y) = \tau_{n+1}(x) <t$ which contradicts our selection of $n$. It follows then that  $\phi(t,y) = \phi_n( y) = \phi_n( x)= \phi(t,x)$. This reasoning also shows $\tau_k(x) = \tau_k(y)$ and
 so $x(\tau_k(x)) = y(\tau_k(y))$ for all $0 \leq k \leq n$; therefore:
\begin{equation}  \label{toBeUsedBelow}
 V(t, x) -  V(t, y) = \phi(t, x) (x(t) - y(t)).
\end{equation}
From the definition it follows that $\phi(t^-,x) = \phi(t,x)$ and notice that  $V(t^-,x)$ exists because $x(t^-)$ exists. It is also straightforward to check that $V(t,x) -V(t^-,x) = \phi(t,x) (x(t) - x(t^-))$ which gives $\psi(t^-, x) = \psi(t, x)$.  Using (\ref{toBeUsedBelow})
we obtain $\psi(t, x)= \psi(t, y)$. It is also straightforward to prove that right limits exist for $\phi$ and $\psi$ as well.
Summarizing, we have argued that $\Phi= (\psi, \phi)$ is NP-predictable.

\noindent
From the definition of $\psi$ it follows that $V_{\Phi}(t,x) = V(t,x)$ for all $t \in [0, T]$ and all $x \in \mathcal{J}$. Therefore, by means
of (\ref{minimumIntegralRequirement}), we obtain
$$
V_{\Phi}(t,x) = V(t,x) = V_{\Phi}(0,x_0) + \int_0^t \phi(s, x)~dx_s~\mbox{for all}~~x \in \mathcal{J},
$$
where $V_{\Phi}(0,x_0) = V_0$, hence $\Phi$ is NP-self financing.

It remains to  check the locally V-continuous property.
For every possible trajectory $x \in \mathcal{J}$, the value of the NP-portfolio $\Phi$ at maturity time $T$  can be expressed as:
\begin{equation}  \nonumber 
V_{\Phi}(T,x)=V_{\Phi}(0,x_0)+\sum_{i=0}^{M(x)-1}\phi_i(x)~\left[ x(\tau_{i+1}(x))-x(\tau_i(x)) \right].
\end{equation}
Consider a fixed, but arbitrary, $x^{\ast} \in \mathcal{J}$,
as $0=\tau_0 \leq \tau_1\leq \tau_2\leq \cdots \leq T$  is a jointly strong locally continuous sequence of NP-stopping times,
there exists an open set $U_{x^*}$ with $x^{\ast} \in \overline{U}_{x^*}$
such that if $x_n \rightarrow x^{\ast}$, with $x_n \in U_{x^*}$, then  $x_n(\tau_{i+1}(x_n))-x_n(\tau_i(x_n)) \rightarrow x^{\ast}(\tau_{i+1}(x^{\ast}))-x^{\ast}(\tau_i(x^{\ast}))$
for all $0 \leq i$. Moreover, as $M(\cdot)$ is integer valued,
$M(x_n) = M(x^{\ast})$ for $n$ large enough by Definition \ref{joint-strong-LC}.  Given that the functions $\phi_0, \ldots, \phi_{M(x^{\ast})}$ are continuous
there are neighborhoods $W_i$ of $x^{\ast}$ such that if $x_n \rightarrow x^{\ast}$, with $x_n \in W_{i}$, then  $\phi_i(x_n) \rightarrow \phi_i(x^{\ast})$
for all $0 \leq i  \leq M(x^{\ast})$. Consider
$W = \cap_{i=0}^{M(x^{\ast})} W_i$ and $V_{x^{\ast}} \equiv W \cap U_{x^*}$. It follows that $x^{\ast} \in \overline{V}_{x^*}$ and hence $V_{x^{\ast}}  \neq \varnothing$, moreover if $x_n \rightarrow x^{\ast}$, with $x_n \in V_{x^{\ast}}$, then
$V_{\Phi}(T,x_n) \to V_{\Phi}(T,x^*)$. Therefore portfolio $\Phi$ is locally V-continuous.
\qed
\end{proof}
The same proof can also be adapted to establish the following corollary.
\begin{corollary}  \label{cor:simplePortfoliosAreLC}
Consider the setting of Theorem \ref{simplePortfoliosAreLC} and assume $\phi_i(x) = \hat{\phi}(x(\tau_i(x)))$ where $\hat{\phi}: \mathbb{R}_+ \rightarrow \mathbb{R}$ is continuous. Then, the
conclusions  of Theorem \ref{simplePortfoliosAreLC} hold.
\end{corollary}
\begin{remark}
The proof of Theorem \ref{simplePortfoliosAreLC} can also be adapted to  cover the case of $\phi(x) = \hat{\phi}(x, \tau_i(x))$ where $\hat{\phi}: \mathcal{J} \times [0,T] \rightarrow \mathbb{R}$ is continuous under the product topology.
\end{remark}

\section{Arbitrage}  \label{arbitrage}

This section provides a high level theorem that allows to transfer no arbitrage results from a standard,  i.e. probabilistic,  setting to a NP setting and vice-versa.
Most of the technical details are implicit in the hypothesis which will need careful consideration in specific instances. The general approach
links probabilistic and NP-portfolios therefore, at this point, we need to introduce some precision about the hypotheses on the probabilistic models that fall under the scope of our results.

The notion of probabilistic market that we use throughout the paper
is similar to the one in \cite{valkeila-2}.
Assume a filtered probability space $(\Omega, \mathcal{F}, (\mathcal{F}_t)_{t \geq 0}, P)$
is given. Let $Z$ be a RCLL adapted stochastic process modelling asset prices defined on this
space.
A portfolio strategy $\Phi^z$ is a
pair of stochastic processes $\Phi^z=(\psi^z, \phi^z)$. The value of
a portfolio $\Phi^z$ at time $t$ is a random variable given by:
\[
V_{\Phi^z}(t)=\psi^z_t+ \phi^z_t Z_t.
\]

A portfolio $\Phi^z$ is self-financing if the integrals $\int_0^t
\psi^z_s(\omega) ds$ and $\int_0^t \phi^z_s(\omega) d Z_s(\omega)$
exist $P$~-a.s. as a F\"ollmer integral
and
\begin{equation} \nonumber
V_{\Phi^z}(t)=V_{\Phi^z}(0) + \int_0^t
\phi^z_s d Z_s, \; P-a.s.
\end{equation}
From now on,  and without further comments,
all  (stochastic) portfolios $\Phi^z$  will be  assumed to be self-financing  and predictable.

\begin{definition} \label{stochasticAdmissiblePortfolio}
A portfolio $\Phi^z$ is admissible if $\Phi^z$ is self-financing, predictable,
and there exists $A^z = A^z(\Phi^z)  \geq 0$ such that $V_{\Phi^z}(t) \geq -~A^z$ $P$~-a.s. $\forall ~t \in [0,T]$.
\end{definition}

\begin{definition}\label{stoch-market}
A stochastic market defined on a filtered probability space
$(\Omega, \mathcal{F}, (\mathcal{F}_t)_{t \geq 0}, P)$ is a
pair $(Z,\mathcal{A}^Z)$ where $Z$ is an adapted stochastic process
modeling asset prices and $\mathcal{A}^Z$ is a class of admissible
portfolio strategies.
\end{definition}
\begin{remark}
We assume $\mathcal{F}_0$ is the trivial sigma algebra, furthermore, without loss of generality, we will assume
that the constant $z_0 = Z(0, w)$ is fixed, i.e. we assume the same initial value for all paths.
The constant $V_{\Phi^z}(0, w)$ will also be denoted $V_{\Phi^z}(0, z_0)$.
\end{remark}

\noindent The notion of arbitrage on a probabilistic market is standard  (in this paper will be referred simply as {\it arbitrage}).
Given a process $Z$ as above, a portfolio $\Phi^z$ is an arbitrage opportunity if: $V_{\Phi^z}(0)=0$ and  $V_{\Phi^z}(T) \geq 0 $,~$P-a.s.$,
and $P(V_{\Phi^z}(T)> 0) >0$. $(Z,\mathcal{A}^Z)$ is arbitrage free if $\Phi^z$ is not an arbitrage, for all $\Phi^z \in \mathcal{A}^Z$.

\vspace{.2in}
\noindent
Given a stochastic process $Z$ and a trajectory space $\mathcal{J}$ as above, consider the
map $Z: \Omega \rightarrow \mathbb{R}^{[0, T]}$ defined by $Z(w)(t)= Z_t(w)$ and
introduce the following two conditions:

\vspace{.05in} \noindent $C_0:$~~~~ $Z(\Omega) \subseteq \mathcal{J}$~a.s.

\noindent
$C_1:$~~~~$Z$ satisfies a small ball property with respect to
the metric $d$ and the space $\mathcal{J}$, namely $\mbox{for all}~~~ \epsilon>0$ and for all $x$ in $\mathcal{J}$:
\[
P\left(d(Z,x)<\epsilon \right)>0.
\]

\begin{definition} \label{isomorphic}
Let a trajectory space $\mathcal{J}$ and a stochastic process $Z$ be given such that condition $C_0$ holds.
A NP-portfolio $\Phi$ defined on $\mathcal{J}$ and a stochastic portfolio $\Phi^z $
are said to be isomorphic if :
\[
P\left( \Phi^z(t,\omega)=\Phi(t,Z(\omega))\; \text{for} \; 0 \leq t \leq T \right)=1.
\]
\end{definition}

The following Theorem is proven similarly to Theorems 1 and 2  from \cite{AFO2011}, we provide the proof for the convenience of the reader.

\begin{theorem} \label{main-arbitrage-local-cont}
Let a trajectory space $\mathcal{J}$ and a stochastic process $Z$ be given such that conditions $C_0$ and $C_1$ hold.
Assume $\Phi$ and $\Phi^z$
are isomorphic, and that, furthermore, $\Phi$ is locally V-continuous, then:\\

\noindent
i) If $\Phi^z$ is not an arbitrage, then $\Phi$ is not a NP-arbitrage portfolio.\\
ii) If $\Phi$ is not a NP-arbitrage, then $\Phi^z$ is not an arbitrage portfolio.
\end{theorem}

\begin{proof}
We proceed to prove $i)$ by contradiction.  Suppose then,
that $\Phi$ satisfies:  $V_{\Phi}(0, x_0) =0$, $V_{\Phi}(T,x)
\geq 0$ for all $ x \in \mathcal{J}$ and there is also $x^{\ast} \in
\mathcal{J}$ such that $V_{\Phi}(T,x^{\ast}) >  0$. Therefore, given that $\Phi^z$ is isomorphic to $\Phi$,
there exists $\Omega_1 \subset \Omega$ with $P(\Omega_1)=1$ such that
$V_{\Phi^z}(0, z_0) = V_{\Phi}(0, x_0) =0$ and $V_{\Phi^z}(T, w) = V_{\Phi}(T, Z(w))$ for
all $w \in \Omega_1$. Define $\Omega_2=\left\{ w \in \Omega: Z(\cdot, \omega) \in \mathcal{J} \right\}$.
Condition $C_0$ implies that $P(\Omega_2)=1$, hence $P(\Omega_1 \cap \Omega_2)=1$.
Consider $w \in \Omega_1 \cap \Omega_2$; then it follows that
$V_{\Phi^z}(T, \omega) \geq 0$, therefore $V_{\Phi^z}(T) \geq 0$ holds $P$-a.s.

Consider $f(x) \equiv  V_{\Phi}(T,x)$ and $\hat{x} \in V_{x^{\ast}}$, where $V_{x^{\ast}}$ is  given as in
Proposition \ref{positiveInOpenSet} (see Appendix). Given that  $V_{x^{\ast}}$ is a nonempty open set, there exists $\epsilon >0$ such that 
$B_\epsilon \equiv \{y: d(y, \hat{x}) < \epsilon\} \subseteq  V_{x^{\ast}}$ and $P(Z(w) \in B_\epsilon) >0$. Therefore
$V_{\Phi^z}(T, \omega) > 0$ on $\{Z(w) \in B_\epsilon\} \cap \Omega_1 \cap \Omega_2$ and this last set has non zero probability; this concludes the proof.

The proof of $ii)$ is similarly achieved by contradiction. Assume $\Phi^z$ is an arbitrage portfolio while $\Phi$ is not.
Notice that $V_{\Phi^z}(0, z_0) = V_{\Phi}(0, x_0)$ and $V_{\Phi^z}(T, w) = V_{\Phi}(T, Z(w))$ a.s.  Assume now that
$V_{\Phi}(T, \hat{x}) <  0$ for some $\hat{x} \in \mathcal{J}$,
local continuity of $V_{\Phi}(T, \cdot)$, an application of Proposition
\ref{positiveInOpenSet} and the small balls property  gives $V_{\Phi}(T, Z(w)) < 0$ for all $w$ in a set of nonzero measure.
This gives, as $\Phi^z$ is isomorphic to $\Phi$, a contradiction, hence $V_{\Phi}(T, x) \geq  0$ for all $x \in \mathcal{J}$.
Moreover, using the isomorphism relationship once more,  the fact that $\Phi^z$ is an arbitrage and $Z(\Omega) \subseteq \mathcal{J}$ a.s.
it follows that there exists $x^{\ast} \in \mathcal{J}$ such that $V_{\Phi}(T, x^{\ast}) >  0$. This is a contradiction and concludes our proof.
\qed
\end{proof}

\section{A Non Probabilistic Jump Diffusion Class}  \label{nPJumpFiffusionClass}

This section defines a realistic trajectory space, denoted $\mathcal{J}_{\mathcal{T}}^{\sigma, C}(x_0)$, and  proves that a large collection of practical NP-portfolios, acting on $\mathcal{J}_{\mathcal{T}}^{\sigma, C}(x_0)$,  are no arbitrage portfolios. Implications to non-semimartingale
stochastic models are also developed.

Given a refining sequence of
partitions  $\mathcal{T}$, 
denote with $\mathcal{Z}_{\mathcal{T}}([0, T])$ the collection of
all continuous functions $z(t)$ such that
$\left[ z \right]_{t}^{\mathcal{T}} = t$ for $0\leq t \leq T$ and
$z(0) =0$. Notice that $\mathcal{Z}_{\mathcal{T}}([0, T])$ includes a.s. paths of
Brownian motion (see \cite{klein}). For this trajectory space, we will assume  that the notion of integral used in (\ref{integrability-assumption}) is the Ito-F\"{o}llmer integral (see \cite{follmer}).


Fix $\sigma > 0$ and $C$ a non empty set of real numbers such
that $\inf(C)>-1$.  Define $\mathcal{J}_\mathcal{T}^{\sigma, C}(x_0)$ as
the class of real valued functions $x$ on $[0,T]$ such that there exits
$z \in \mathcal{Z}_{\mathcal{T}}([0, T])$, $n(t) \in \mathcal{N}([0,
T])$ (this last class has been introduced in Section \ref{locallyContinuousPortfolios}),  and real numbers $a_i \in C$, $i=1,2,\ldots,m$, verifying:
\begin{equation} \label{eq:diffpoisson}
x(t)=x_0 e^{\sigma z(t)}\prod_{i=1}^{n(t)} (1+a_i),
\end{equation}
where $n(t)$ was introduced in (\ref{eq:poisson}).

\subsection{Locally V-Continuous Portfolios on $\mathcal{J}_{\tau}^{\sigma, C}(x_0)$}

We call attention to an implication of Lemma \ref{Lemma-correspondence-jumps}, the statement and proof of which can be found in the Appendix,
showing that for $n$ large enough the number of jumps in trajectory $x_n$
between two consecutive stopping times $\tau_i(x_n)$ and $\tau_{i+1}(x_n)$ is the same as the
number of jumps in trajectory $x^*$ between two consecutive stopping times $\tau_i(x^*)$ and  $\tau_{i+1}(x^*)$ (where $x_n$ and $x^{\ast}$ are as in Lemma \ref{Lemma-correspondence-jumps}). This result will be used in the following theorem.

\begin{theorem}  \label{continuousRebalancingPortfolios}
Let $0=\tau_0 \leq \tau_1\leq \tau_2\leq \cdots \leq T$  be a jointly strong locally continuous sequence
of NP-stopping times (as per Definition \ref{joint-strong-LC}) defined on $\mathcal{J}_\mathcal{T}^{\sigma, C}(x_0)$ with respect to the Skorohod's metric. Assume that $\inf_{c \in C}|c|>0$.
Let $\phi_0(\cdot,\cdot), \phi_1(\cdot,\cdot),... $ be functions continuous on $[0,T]\times \mathbb{R}$ and differentiable on $(0,T)\times \mathbb{R}$. Consider
the portfolio strategy given by $\Phi_t=(\psi_t, \phi_t)$ where the amount invested in the stock
$\phi_t$ is such that
\[
\phi(t,x)=1_{[\tau_0, \tau_{1}]}(t)\phi_0(t,x(t^{-}))+\sum_{i=1}^{M(x)-1} 1_{(\tau_i, \tau_{i+1}]}(t)\phi_i(t,x(t^{-})),
\]
and $\psi_t$ is given as described in Remark \ref{selfFinancingBankInvestment}.
Then, the portfolio $\Phi$ is NP-predictable, NP-self-financing and locally V-continuous on $\mathcal{J}_\mathcal{T}^{\sigma,C}$ relative to the
Skorohod's topology.
\end{theorem}

\begin{proof}
Using similar arguments to the ones used in Theorem \ref{simplePortfoliosAreLC} one can prove that
a portfolio $\Phi$ as above is NP-predictable and NP-self-financing. Next we will prove
that it is also locally V-continuous.\\

For $i=0,1,...$ define the functions $U^i_{\phi}: \mathbb{R}^2 \rightarrow \mathbb{R}$ as:
\begin{equation}  \nonumber
U^i_{\phi}(t,y)=\int_{y_0}^y \phi_i(t, \xi) d \xi.
\end{equation}
Let
\begin{equation}  \nonumber
u_{\Phi}(x)=\sum_{i=0}^{M(x)-1} u_{\Phi}^i(x),
\end{equation}
where the functionals $u_{\Phi}^i: \mathcal{J}_\mathcal{T}^{\sigma,C} \rightarrow \mathbb{R}$ are defined as:
\begin{eqnarray}\label{functional-u1}
u_{\Phi}^i(x)&=&U_{\Phi}^i(\tau_{i+1}(x), x(\tau_{i+1}(x)))-U_{\Phi}^i(\tau_i(x), x(\tau_i(x)))\\
           &&-\int_{\tau_i(x)}^{\tau_{i+1}(x)} \frac{\partial U_{\Phi}^i}{\partial t}(s,x(s-))ds
             -\frac{1}{2}\int_{\tau_i(x)}^{\tau_{i+1}(x)} \frac{\partial^2 U_{\Phi}^i}{\partial x^2}(s,x(s-))d \langle x \rangle_s^{\mathcal{T}} \nonumber \\
           &&-\sum_{\tau_i(x) < s \leq \tau_{i+1}(x)} \left[ U_{\Phi}^i(s,x(s))-U_{\Phi}^i(s,x(s-))
             -\frac{\partial U_{\Phi}^i}{\partial x}(s,x(s-))\Delta x(s)\right]. \nonumber
\end{eqnarray}

\noindent
The It\^o-F\"ollmer formula from \cite{follmer} allows us to obtain:

\begin{equation} \nonumber
u_{\Phi}^i(x)=\int_{\tau_i(x)}^{\tau_{i+1}(x)} \frac{\partial U_{\Phi}^i}{\partial x}(s,x(s-)) dx(s)=\int_{\tau_i(x)}^{\tau_{i+1}(x)} \phi_i(s,x(s-)) dx(s),
\end{equation}
then
\begin{equation} \nonumber
u_{\Phi}(x)=\sum_{i=0}^{M(x)-1} u_{\Phi}^i(x)= \int_{0}^{T} \phi(s,x(s-)) dx(s).
\end{equation}

\noindent
For all $x \in \mathcal{J}_\mathcal{T}^{\sigma,C}$, $d \langle x \rangle_s^{\mathcal{T}}= \sigma^2 x^2(s-)ds$,
therefore:
\begin{equation*}
u_{\Phi}^i(x)=U_{\Phi}^i(\tau_{i+1}(x), x(\tau_{i+1}(x)))-U_{\Phi}^i(\tau_i(x), x(\tau_i(x)))-I_{\Phi}^i(x)-S_{\Phi}^i(x)
\end{equation*}
where
\begin{equation}\nonumber 
I_{\Phi}^i(x)=\int_{\tau_i(x)}^{\tau_{i+1}(x)}  \frac{\partial U_{\Phi}^i}{\partial t}(s,x(s-))ds
             +\frac{1}{2}\int_{\tau_i(x)}^{\tau_{i+1}(x)}  \frac{\partial^2 U_{\Phi}^i}{\partial x^2}(s,x(s-))\sigma^2 x^2(s-)ds
\end{equation}
and
\begin{equation} \nonumber
S_{\Phi^i}(x)=\sum_{\tau_i(x) < s \leq \tau_{i+1}(x)} \left[ U_{\Phi}^i(s,x(s))-U_{\Phi}^i(s,x(s-))
            -\frac{\partial U_{\Phi}^i}{\partial x}(s,x(s-))\Delta x(s)\right].
\end{equation}

Fix $x^* \in \mathcal{J}_\mathcal{T}^{\sigma, C}(x_0)$;
as $0=\tau_0 \leq \tau_1\leq \tau_2\leq \cdots \leq T$  is a jointly strong locally continuous sequence
of NP-stopping times, there exists an open set $U_{x^*} \subset \mathcal{J}_\mathcal{T}^{\sigma, C}(x_0)$
such that $x^* \in \overline{U}_{x^*}$ and whenever $U_{x^*} \ni x_n \to x^*$, then, properties   
i), ii) and iii) from Definition \ref{joint-strong-LC} hold. Let $\{x_n\}_{n \geq 1}$ be such a sequence.
We have the following:\\

\noindent
1) Using that $U_{\Phi}^i$ is continuous for all $i$, and the jointly strong local continuity property of 
$\left\{ \tau_n \right \}_{n=0,1, ...}$ we can check that 

$$U_{\Phi}^i(\tau_{i+1}(x_n), x_n(\tau_{i+1}(x_n)))-U_{\Phi}^i(\tau_i(x_n), x_n(\tau_i(x_n)))$$
converges to 

$$U_{\Phi}^i(\tau_{i+1}(x^*), x^*(\tau_{i+1}(x^*)))-U_{\Phi}^i(\tau_i(x^*), x^*(\tau_i(x^*)))$$
as $n$ approaches infinity.\\

\noindent
2) Also, $I_{\Phi}^i(x_n) \to I_{\Phi}^i(x^*)$ as $n$ approaches infinity. This can be proved 
using the same technique as in the proof of Proposition 7 in \cite{AFO2011}.   \\

\noindent
3) Along the lines of the proof of Proposition 7 in \cite{AFO2011} we can also prove that 
$S_{\Phi^i}(x_n) \to S_{\Phi^i}(x^*)$ as $n$ approaches infinity. The only new element in the proof 
is the use of Lemma \ref{Lemma-correspondence-jumps} in order to establish a correspondence between
the jumps of $x_n$ and those of $x^*$ for $n$ large enough. 

Combining 1), 2) and 3) above we get that for all $i$, $u_{\Phi}^i(x_n)$ converges to $u_{\Phi}^i(x^*)$,
therefore $u_{\Phi}(x_n)$ converges to $u_{\Phi}(x^*)$. 

\noindent This implies that

\[
V_{\Phi}(T,x_n)=V_0+ \int_{0}^{T} \phi(s,x_n(s-)) dx_n(s) \to V_0+ \int_{0}^{T} \phi(s,x^*(s-)) dx^*(s)=V_{\Phi}(T,x^*)
\]
therefore $\Phi$ is locally V-continuous on $\mathcal{J}_\mathcal{T}^{\sigma,C}$ relative to the
Skorohod's topology. \qed
\end{proof}

The following proposition provides examples of sequences of NP-stopping times that are jointly strong locally continuous
on $\mathcal{J}_\mathcal{T}^{\sigma, C}(x_0)$.

\begin{proposition} \label{pr:stop-tim-ex}
Let $\{K_i\}_{i=1,2,\ldots}$ be an increasing sequence of real numbers with $K_i \to \infty$, and $K_i>x_0$ for all $i$. If $\inf_{c \in C} |c|>0$ then
the following sequences of NP-stopping times are jointly strong locally continuous on $\mathcal{J}_\mathcal{T}^{\sigma, C}(x_0)$ with respect to the Skorohod's metric:
\begin{itemize}
\item 1) $\tau_i(x)=\min(\frac{i~T}{N}, T)$, for $i=0,1, \ldots,~\mbox{and}~N \geq 1~\mbox{an arbitrary integer}$.
\item 2) $\displaystyle{\tau_i(x)=\min \left(\inf \left\{t: \sum_{s \leq t} 1_{\mathbb{R} \setminus \{0\} } (x(s)-x(s-))\geq i\right\}, T\right) }$, for $i=1,2, \ldots$
\item 3) $\displaystyle{\tau_i(x)=\min \left( \inf \{t:  x_t \geq K_i\}, T \right)}$, for $i=1,2, \ldots$

\end{itemize}
\end{proposition}

\begin{proof} See the Appendix.
\end{proof}

\subsection{Arbitrage-Free NP-Portfolios for Jump Diffusion Class $\mathcal{J}_\mathcal{T}^{\sigma, C}$}

This section proves a class of NP-portfolios to be NP-arbitrage free for the trajectory space $\mathcal{J}_\mathcal{T}^{\sigma, C}$.
Towards this end we will make use of Theorem \ref{main-arbitrage-local-cont} which, in turns, requires the introduction of an appropriate stochastic market model.

\begin{definition} \label{jumpDiffusionClass}
For any $x_0>0$ consider, on a probability space
$(\Omega, \mathcal{F}, (\mathcal{F}_t)_{t \geq 0}, P)$,  an  exponential jump
diffusion processes, starting at $x_0$ given by:
\begin{equation} \nonumber
     Z_t=x_0 ~e^{(\mu-\frac{1}{2}\sigma^2)t+\sigma W_t}\prod_{i=1}^{N_t} (1+X_i),
\end{equation}
where $W= \{W_t\}$ is a standard Brownian motion, $N=  \{N_t\}$ is a homogeneous Poisson Process  with intensity
$\lambda >0$,
and the $X_i$ are independent random variables, also
independent of $W$ and $N$, with common probability distribution
$F_X$.
\end{definition}

\vspace{.1in}
\noindent
Let $\mathcal{A}^Z_{JD}$ be the class of admissible strategies (as in Definition \ref{stochasticAdmissiblePortfolio}) for the process $Z$.

\vspace{.1in}
\noindent
The following theorem makes use of notation introduced above.
\begin{theorem}\label{noArbitrageHedginStrategies}
Let $\mathcal{J}_{\mathcal{T}}^{\sigma, C}$ be the trajectory class
introduced in (\ref{eq:diffpoisson}) endowed with the Skorohod's topology.
Assume the random variables $X_i$  to be integrable with common probability distribution $F_X$ satisfying:\\
1) $P (X_i \subset C)=1$. \\
2) For any $a \in C$ and for all $\epsilon >0$, $F_X(a+\epsilon)-F_X(a-\epsilon)>0$.\\
Let $\Phi$ denote one of the portfolios considered in Corollary \ref{cor:simplePortfoliosAreLC}  or Theorem \ref{continuousRebalancingPortfolios}
defined through the NP-stopping times from Proposition \ref{pr:stop-tim-ex}. These are NP-portfolios which we require to be
NP-admissible as per Definition \ref{npAdmissiblePortfolio}. Then, such a $\Phi$ is not a NP-arbitrage portfolio.
\end{theorem}

\begin{proof}
We will apply Theorem \ref{main-arbitrage-local-cont} to $Z$ and $\mathcal{J}_\mathcal{T}^{\sigma, C}$.
Note that $P(w \in \Omega : Z(w) \in \mathcal{J}_\mathcal{T}^{\sigma, C})=1$, this follows from our assumption 1). Therefore, hypothesis $C_0$ in Theorem \ref{main-arbitrage-local-cont} is fulfilled.
Our assumption 2) allows for the application of  Proposition 6 in \cite{AFO2011}, therefore, we conclude that the process $Z$ satisfies a small ball property with respect to
Skorohod's metric and trajectory space $\mathcal{J}_\mathcal{T}^{\sigma, C}$. It follows then that hypothesis $C_1$ in Theorem \ref{main-arbitrage-local-cont} is fulfilled as well.

\noindent
Let $\Phi$ be one of the portfolios described in the statement of the theorem and define
\begin{equation} \label{portfoliosAreTheSameOnPathSpace}
\Phi^z(t, w) = \Phi(t, Z(w)),
\end{equation}
notice that (\ref{portfoliosAreTheSameOnPathSpace}) is well defined in a set of full measure. We will argue below that
$\Phi^z \in \mathcal{A}_{JD}^Z$; notice that (\ref{portfoliosAreTheSameOnPathSpace}) shows $\Phi$ to be isomorphic to $\phi^z$.
Our hypothesis on the process $Z$ allow to apply
Proposition 9.9 from \cite{cont}, this result establishes the existence of a probability $\mathbb{Q}$ such that $e^{-r t}~Z_t$
is a martingale and so the probabilistic market
$(Z,\mathcal{A}_{JD}^Z)$ is arbitrage free. Therefore, elements of 
$ \mathcal{A}_{JD}^Z$ are not  arbitrage portfolios.
It then follows from Theorem \ref{main-arbitrage-local-cont}, statement i), that $\Phi$ is not a NP-arbitrage portfolio.

To complete the argument it remains to prove that $\Phi^z \in \mathcal{A}_{JD}^Z$, this is equivalent to proving that $\Phi^z$, as given by (\ref{portfoliosAreTheSameOnPathSpace}), is admissible as per Definition \ref{stochasticAdmissiblePortfolio}.
Notice that $\Phi^z$ is LCRL because $\Phi$ is LCRL.
Given that $\Phi$ is a NP-portfolio, assumed to be NP-admissible, it then follows that
to show admissibility of $\Phi^z$ it is enough to show that $\Phi^z_t$ is a
predictable process. We provide the proof of this fact only for the stock component $\phi^z_t$; because of the left continuity property,  $\phi^z_t$  will be predictable if it is adapted to the given filtration
$\mathcal{F} = \{\mathcal{F}_t\}$, we prove this next. 
Let $\tau$ denote one of the NP-stopping times considered in the statement of the theorem and define $\hat{\tau}(w)= \tau(Z(w))$, this maps is defined on a set of full measure
and it is easy to show that they are stopping times with respect to $\{\mathcal{F}_t\}$. In particular, the  simple portfolios have the form:
\begin{equation} \label{simpleStochasticPortfolio}
\phi^z(t,w) \equiv \phi^z_{0}(w)~ 1_{[0, \hat{\tau}_{1}(w)]}(t)+ \sum_{k \geq 1}~ \hat{\phi}_k(Z_{\hat{\tau}_k(w)}(w)) 1_{(\hat{\tau}_k(w), \hat{\tau}_{k+1}(w)]}(t),
\end{equation}
where $\hat{\phi}: \mathbb{R}_+ \rightarrow \mathbb{R}$ is continuous and hence $\hat{\phi}_k(Z_t(w))$ is $\mathcal{F}_t$-measurable. It follows that
(\ref{simpleStochasticPortfolio}) is $\mathcal{F}_t$-measurable. A similar argument also can be used for the stochastic portfolios
isomorphic to the continuing re-balancing portfolios from Theorem \ref{continuousRebalancingPortfolios}.
\qed
\end{proof}
A more general result can actually be proven as well.
\begin{corollary} \label{cor:noArbitrageNPMarket}
Assume the same hypothesis as in Theorem \ref{noArbitrageHedginStrategies} but now consider the following class of portfolios:
\begin{equation} \nonumber
\mathcal{A} \equiv \{\Phi: \Phi~\mbox{is a NP-portfolio, locally V-continuous and}~\exists
 ~\Phi^z \in \mathcal{A}^Z_{JD} ~\mbox{isomorphic to}~ \Phi\}.
\end{equation}
Then, the NP-market $(\mathcal{J}, \mathcal{A})$ is NP-arbitrage free.
\end{corollary}
The proof of Corollary \ref{cor:noArbitrageNPMarket} is exactly the same as the one of Theorem \ref{noArbitrageHedginStrategies}; the point of the specialized
Theorem \ref{noArbitrageHedginStrategies} is to  explicitly establish membership to $\mathcal{A}$ for the portfolios considered in our paper.
Theorem 7 from \cite{AFO2011} provides further examples of portfolios belonging to $\mathcal{A}$.

\subsection{Implications to Stochastic Frameworks}  \label{implicationsForStochasticFrameworks}

Theorem \ref{main-arbitrage-local-cont}, item $ii)$, in conjunction with Corollary \ref{cor:noArbitrageNPMarket}, can be used
to prove that certain stochastic models are arbitrage free. A main point to emphasize is that many of these stochastic models are not semi-martingales, moreover, the NP-portfolios defined through NP-stoping-times considered in the present paper define isomorphic stochastic portfolios in such models. Below, we provide the main steps required to obtain these type of results and refer to  \cite{AFO2011} for more details.

\begin{example}[Jump-diffusion related models] \label{ex:mixed-jump-diff} Consider the following  stochastic process, defined on a filtered space $(\Omega, \{\mathcal{F}_t\}, P)$,
\begin{equation} \nonumber
Y_t=e^{(\mu-\sigma^2/2)t + \sigma Z^G_t}\prod_{i=1}^{N^R_t}(1+Y_i),
\end{equation}
where $Z^G$ is a continuous process satisfying $\langle Z^G\rangle_t=t$. Assume also
that $Z^G$ satisfies a small ball property on $\mathcal{Z}_{\mathcal{T}}([0, T])$
with respect to the uniform norm. Examples of such processes $Z^G$ are the processes $Z^F$, $Z^R$ and $Z^w$ introduced
in Section 5 of \cite{AFO2011}. The process $N^R$ is a renewal process and the random variables $Y_i$ are independent and also independent of $Z^G$ and $N^R$  with
common distribution $F_Y$.

\noindent
Consider the arbitrage free NP-market
$(\mathcal{J}_\mathcal{T}^{\sigma, C}, \mathcal{A})$ introduced  in Corollary \ref{cor:noArbitrageNPMarket} and define the following set of portfolios  (defined on $(\Omega, \{\mathcal{F}_t\}, P)$),
\begin{equation}  \nonumber
\mathcal{A}^Y \equiv \{\Phi^y:~\mbox{admissible and}~\exists~\Phi \in \mathcal{A}~\mbox{isomorphic to} ~\Phi^y\}.
\end{equation}
We argue next that, under appropriate conditions, the stochastic market $(Y, \mathcal{A}^Y )$ is arbitrage free.
Assume the hypothesis in Corollary \ref{cor:noArbitrageNPMarket}
are satisfied hence $(\mathcal{J}_\mathcal{T}^{\sigma, C}, \mathcal{A})$ is NP-arbitrage free. Furthermore, under the assumptions:\\
1) $P (Y_i \subset C)=1$, \\
2) For any $a \in C$ and for all $\epsilon >0$, $F_Y(a+\epsilon)-F_Y(a-\epsilon)>0$,\\
one can use the same arguments as in the proof of Theorem \ref{noArbitrageHedginStrategies} to show that
hypothesis $C_0$ and $C_1$ in Theorem \ref{main-arbitrage-local-cont} hold; therefore, using
ii) from that latter theorem one concludes that  $(Y, \mathcal{A}^Y)$ is arbitrage free.
\end{example}

\section{Variable Volatility Models}  \label{variableVolatilityClass}

Analogously to the developments in Section  \ref{nPJumpFiffusionClass},
the present section defines a class of trajectories $J_{\mathcal{T}}^{\Sigma}(x_0)$. 
Notice that the refining sequence of partitions $\mathcal{T}$ has been introduced in Section \ref{nPJumpFiffusionClass}.
The set $J_{\mathcal{T}}^{\Sigma}(x_0)$ exhibits different volatilities for different trajectories, that is, the volatility curve/function is trajectory-dependent. A new metric is introduced that allows to prove 
that a large class of practical NP-portfolios are arbitrage free.
Moreover, we draw some no-arbitrage implications for modified stochastic Heston models which include non-semimartingale processes. As in the previous section, the notion of integral that will be used throughout this Section is the Follmer's integral.

Let $\Sigma \subset C[0,T]$ be a class of functions of finite variation representing the possible volatility trajectories.  Also, let  $NQV[0,T]\subset C[0,T]$ be the class of all continuous functions $d:[0,T] \to \mathbb{R}$ that have null quadratic variation
on $[0,T]$ and satisfy $d_0=d(0)=0$.\\
\\
Define
\begin{equation} \nonumber 
J_{\mathcal{T}}^{\Sigma}(x_0)= \left\{x \in C[0,T]: x(t)=x_0e^{d_t+\int_0^t \sigma(s) dz(s)}\;, \sigma \in  \Sigma, z \in \mathcal{Z}_{\mathcal{T}}([0, T]), d \in NQV[0,T] \right\}
\end{equation}
The integral that appears in the previous expression exists as $\sigma$ has finite variation. For a more detailed discussion on the existence of these integrals 
see \cite{Schied}.\\
\\
Consider now a metric on $J_{\mathcal{T}}^{\Sigma}(x_0)$ given by
\begin{equation} \nonumber
d_{QV}(x,y)=\left\|x-y \right\|+\left\|\frac{\partial}{ \partial t} \left\langle x\right\rangle_t -
                                  \frac{\partial}{ \partial t} \left\langle y\right\rangle_t \right\|,
\end{equation}
where $\left\| \cdot\right\|$ stands for the supremum norm on $C[0,T]$.

\begin{remark}
Using Ito-Follmer's formula (see \cite{follmer}) we can check that if
$x_t=x_0e^{d^x_t+\int_0^t \sigma_x(s) dz_x(s)}$ and $y_t=x_0e^{d^y_t+\int_0^t \sigma_y(s) dz_y(s)}$
are two trajectories  in $J_{\mathcal{T}}^{\Sigma}(x_0)$, then
\[
d_{QV}(x,y)=\left\|x-y \right\|+\left\|x^2\sigma_x^2 - y^2\sigma_y^2\right\|.
\]
This means that $x$ and $y$ will be close in the metric $d_{QV}$ if they are close
in the uniform metric and their volatilities are also close in the uniform metric.
\end{remark}

The following proposition gives sufficient conditions in order to establish the small
balls property of some stochastic volatility models on $J_{\mathcal{T}}^{\Sigma}(x_0)$
with respect to the metric $d_{QV}$. 

\begin{proposition}\label{prop:brownian-dep}
Let $\Sigma \subset C[0,T]$ be a set of strictly positive functions of finite variation.
Let $Z$ be a stochastic volatility model on $(\Omega, \mathcal{F}, (\mathcal{F}_t)_{t \geq 0}, P)$ given by
$Z_t=x_0e^{h_t+\int_0^t \sigma_s dW_s}$ where $W$, $h$ and $\sigma$ are stochastic processes.
The stochastic process $h$ is also assumed to have null quadratic variation and $h_0=0$.
Assume that
$P \left( \sigma(\omega) \in \Sigma\right) =1$, and $\sigma$ satisfies a small balls property on $\Sigma$ with respect to the uniform norm. Assume also that $P \left( W(\omega) \in \mathcal{Z}_{\mathcal{T}}([0, T])\right) = 1$, and there exists $0 < \alpha \leq 1$
such that $W=\alpha B+Y$ where $B$ is a Brownian motion independent of $Y$, $\sigma$ and $h$. Then:\\
\\
i) $P \left( Z(\cdot,\omega) \in J_{\mathcal{T}}^{\Sigma}(x_0)\right) = 1$.\\
ii) For all $y \in J_{\mathcal{T}}^{\Sigma}(x_0)$ and for all $\epsilon>0$,
$P\left(d_{QV}\left(Z(\omega),y \right)<\epsilon \right)>0$.
\end{proposition}
\begin{proof}
The proof of i) is immediate from the construction of $Z$. To prove statement ii), notice that if $y_t=x_0e^{d^y_t+\int_0^t \sigma_y(s) dz_y(s)}$ then

\[
\left\{\omega: d_{QV}\left(Z(\cdot,\omega),y \right)<\epsilon \right\} \supset A \cap B
\]
where A and B are defined as
\[
A=\left\{\omega: ||Z(\cdot,\omega)-y||< \frac{\epsilon}{2} \right\}
\]
and
\[
B=\left\{ \omega: \left\|\frac{\partial}{ \partial t} \left\langle Z(\cdot,\omega)\right\rangle_t -\frac{\partial}{ \partial t} \left\langle y\right\rangle_t \right\| < \frac{\epsilon}{2} \right\}
=\left\{ \omega: \| Z^2(\cdot,\omega) \sigma^2(\omega)-y^2\sigma_y^2\| < 
\frac{\epsilon}{2} \right\}
\]
The probability $P(A)>0$ as consequence of Theorem 3.1 in \cite{pak} (see also  Remark 4.3 in \cite{pak}). Here we used the independence between $Y$ and $B$.
The conditional probability $P(B|A)$ is also positive as consequence of the small balls property of
$\sigma$ on $\Sigma$ with respect to the uniform norm. Then, from $P(A\cap B)=P(B|A)P(A)$, we conclude that $P(A \cap B)>0$, therefore $P \left(\omega: d_{QV}\left(Z(\cdot,\omega),y \right)<\epsilon \right) >0$, for all $y \in J_{\mathcal{T}}^{\Sigma}(x_0)$ and for all $\epsilon>0$.
\qed
\end{proof}

\begin{remark} Similar results using a general integrator $W$ could be obtained by assuming independence between $\sigma$ and $W$, see \cite{pak} for some related results. However, from the modeling point of view, it is not desirable that $\sigma$ and $W$ are independent.
\end{remark}

\subsection{Locally V-Continuous Portfolios on $J_{\mathcal{T}}^{\Sigma}(x_0)$}

The following theorem establishes that a large class of portfolios acting 
on $J_{\mathcal{T}}^{\Sigma}(x_0)$ are locally $V$-continuous.

\begin{theorem}\label{loc-cont-portfolio-under-QV}
Let $0=\tau_0 \leq \tau_1\leq \tau_2\leq \cdots \leq T$  be a jointly strong locally continuous sequence
of NP-stopping times in $J_{\mathcal{T}}^{\Sigma}(x_0)$ with respect to the  metric $d_{QV}$.
Let $\phi_0(\cdot,\cdot), \phi_1(\cdot,\cdot),...$
be functions continuous on $[0,T]\times \mathbb{R}$ and differentiable on $(0,T)\times \mathbb{R}$. Consider
the portfolio strategy given by $\Phi_t=(\psi_t, \phi_t)$ where the amount invested in the stock
$\phi_t$ is such that
\[
\phi(t,x)=1_{[\tau_0, \tau_{1}]}(t)\phi_0(t,x(t^{-}))+\sum_{i=1}^{M(x)-1} 1_{(\tau_i, \tau_{i+1}]}(t)~\phi_i(t,x(t^{-})),
\]
and $\psi_t$ is given as described in Remark \ref{selfFinancingBankInvestment}.
Then, the portfolio $\Phi$ is NP-predictable, NP-self-financing and locally V-continuous on $J_{\mathcal{T}}^{\Sigma}(x_0)$ relative to the
metric $d_{QV}.$
\end{theorem}

\begin{proof}
Using similar arguments to the ones used in Theorem \ref{simplePortfoliosAreLC} one can prove that
a portfolio $\Phi$ as above is NP-predictable and NP-self-financing. Next we will prove
that it is also locally V-continuous.\\

For $i=0,1,...$ define the function $U^i_{\phi}: \mathbb{R}^2 \rightarrow \mathbb{R}$ as:
\begin{equation} \nonumber
U^i_{\phi}(t,y)=\int_{y_0}^y \phi_i(t, \xi) d \xi
\end{equation}
Let
\begin{equation} \nonumber
u_{\Phi}(x)=\sum_{i=0}^{M(x)-1} u_{\Phi}^i(x)
\end{equation}
where the functionals $u_{\Phi}^i: J_{\mathcal{T}}^{\Sigma}(x_0) \rightarrow \mathbb{R}$ are defined as:
\begin{eqnarray}\label{eq:stoch-vol-0}
u_{\Phi}^i(x)&=&U_{\Phi}^i(\tau_{i+1}(x),x(\tau_{i+1}(x)))-U_{\Phi}^i(\tau_i(x),x(\tau_i(x)))\\
           &&-\int_{\tau_i(x)}^{\tau_{i+1}(x)} \frac{\partial U_{\Phi}^i}{\partial t}(s,x(s-))ds
             -\frac{1}{2}\int_{\tau_i(x)}^{\tau_{i+1}(x)} \frac{\partial^2 U_{\Phi}^i}{\partial x^2}(s,x(s-))d \langle x \rangle_s^{\mathcal{T}}. \nonumber
\end{eqnarray}

From It\^o-F\"ollmer formula

\begin{equation} \nonumber
u_{\Phi}^i(x)=\int_{\tau_i(x)}^{\tau_{i+1}(x)} \frac{\partial U_{\Phi}^i}{\partial x}(s,x(s-)) dx(s)=\int_{\tau_i(x)}^{\tau_{i+1}(x)} \phi_i(s,x(s-)) dx(s).
\end{equation}
Then
\begin{equation} \label{eq:stoch-vol-sum}
u_{\Phi}(x)=\sum_{i=0}^{M(x)-1} u_{\Phi}^i(x)= \int_{0}^{T} \phi(s,x(s-)) dx(s).
\end{equation}

Now fix $x^* \in J_{\mathcal{T}}^{\Sigma}(x_0)$.
As $0=\tau_0 \leq \tau_1\leq \tau_2\leq \cdots \leq T$  is a jointly strong locally continuous sequence
of NP-stopping times, there exists an open $U_{x^*}\subset J_{\mathcal{T}}^{\Sigma}(x_0)$
such that $x^* \in \overline{U}_x$ and whenever $x_n \to x^*$ in $U_{x^*}$, i), ii) and iii) of Definition
\ref{joint-strong-LC} hold.\\
\\
Using that $U(\cdot,\cdot) \in C^{1,2}([0,T] \times \mathbb{R})$, the continuity of $x_n$ and $x^*$, and also that
$\tau_i(x_n) \to \tau_i(x^*)$ for all $i$, we conclude that:\\
\begin{equation}\label{eq:stoch-vol-1}
U_{\Phi}^i(\tau_{i+1}(x_n),x_n(\tau_{i+1}(x_n)))-U_{\Phi}^i(\tau_i(x_n),x_n(\tau_i(x_n))) \to
\end{equation}
\begin{equation} \nonumber
 U_{\Phi}^i(\tau_{i+1}(x^*),x^*(\tau_{i+1}(x^*)))-U_{\Phi}^i(\tau_i(x^*),x^*(\tau_i(x^*)))
\end{equation}
and
\begin{equation}\label{eq:stoch-vol-2}
\int_{\tau_i(x_n)}^{\tau_{i+1}(x_n)} \frac{\partial U_{\Phi}^i}{\partial t}(s,x_n(s-))ds \to \int_{\tau_i(x^*)}^{\tau_{i+1}(x^*)} \frac{\partial U_{\Phi}^i}{\partial t}(s,x^*(s-))ds.
\end{equation}

\noindent On the other hand, we have that
$\displaystyle{d \langle x_n \rangle_s^{\mathcal{T}} =\frac{d \langle x_n \rangle_s^{\mathcal{T}}}{ds} ds}$ and
$\displaystyle{d \langle x^* \rangle_s^{\mathcal{T}} =\frac{d \langle x^* \rangle_s^{\mathcal{T}}}{ds} ds}$.

\noindent The convergence of $x_n$ to $x^*$ in the metric $d_{QV}$ implies that
$\displaystyle{\frac{d \langle x_n \rangle_s^{\mathcal{T}}}{ds} \to \frac{d \langle x^* \rangle_s^{\mathcal{T}}}{ds}}$
uniformly on $[0,T]$. This, together with the fact that $U(\cdot,\cdot) \in C^{1,2}([0,T] \times \mathbb{R})$ and the
convergence of $\tau_i(x_n)$ to $\tau_i(x^*)$ for all $i$, imply that\\

$\displaystyle{\int_{\tau_i(x_n)}^{\tau_{i+1}(x_n)}
\frac{\partial^2 U_{\Phi}^i}{\partial x^2}(s,x_n(s-))\frac{d \langle x_n \rangle_s^{\mathcal{T}}}{ds} ds \to
\int_{\tau_i(x^*)}^{\tau_{i+1}(x^*)}
\frac{\partial^2 U_{\Phi}^i}{\partial x^2}(s,x^*(s-))\frac{d \langle x^* \rangle_s^{\mathcal{T}}}{ds} ds}$

\noindent or equivalently

\begin{equation}\label{eq:stoch-vol-3}
\int_{\tau_i(x_n)}^{\tau_{i+1}(x_n)} \frac{\partial^2 U_{\Phi}^i}{\partial x^2}(s,x_n(s-))d \langle x_n \rangle_s^{\mathcal{T}} \to
\int_{\tau_i(x^*)}^{\tau_{i+1}(x^*)} \frac{\partial^2 U_{\Phi}^i}{\partial x^2}(s,x^*(s-))d \langle x^* \rangle_s^{\mathcal{T}}.
\end{equation}

\noindent Combining expressions (\ref{eq:stoch-vol-1}), (\ref{eq:stoch-vol-2}) and (\ref{eq:stoch-vol-3}) with
(\ref{eq:stoch-vol-0}) and (\ref{eq:stoch-vol-sum}) we get

\[
u_{\Phi}(x_n) \to u_{\Phi}(x^*).
\]

\noindent This implies that

\[
V_{\Phi}(T,x_n)=V_0+ \int_{0}^{T} \phi(s,x_n(s)) dx_n(s) \to V_0+ \int_{0}^{T} \phi(s,x^*(s)) dx^*(s)=V_{\Phi}(T,x^*),
\]
\noindent 
so $\Phi$ is locally V-continuous on $J_{\mathcal{T}}^{\Sigma}(x_0)$ relative to the
metric $d_{QV}$
\qed
\end{proof}

The following proposition provides examples of sequences of NP-stopping times that are jointly strong locally continuous
on $J_{\mathcal{T}}^{\Sigma}(x_0)$ and, hence,  examples of NP-stopping times
satisfying the hypothesis required in Theorem \ref{loc-cont-portfolio-under-QV}.

\begin{proposition}\label{pr:stoppingtimes-stoch-vol}
Let $\{K_i\}_{i=1,2,\ldots}$ be an increasing sequence of real numbers with $K_i \to \infty$. The following sequences of NP-stopping times are jointly strong locally continuous in $J_{\mathcal{T}}^{\Sigma}$ with respect to the metric $d_{QV}$:
\begin{itemize}
\item 1) $\tau_i(x)=\min(\frac{i~T}{n}, T)$, for $i=0,1, \ldots$
\item 2) $\displaystyle{\tau_i(x)=\min \left(\inf \{t :  x_t \geq K_i\}, T \right)}$, for $i=1,2, \ldots$
\end{itemize}
\end{proposition}

\begin{proof}
Fix $x^* \in J_{\mathcal{T}}^{\Sigma}$ and define:

\begin{equation*} \nonumber
U^{1, \epsilon}_{x^*}=\left \{ y \in J_{\mathcal{T}}^{\Sigma}: 0< d_{QV}(y,x^*)< \epsilon \right \},
\end{equation*}

\begin{equation*} \nonumber
U^{2, \epsilon}_{x^*}=\left \{ y \in J_{\mathcal{T}}^{\Sigma}:  y(t) > x^*(t) \; \text{for}  \; t \geq \epsilon \right \}.
\end{equation*}

\begin{equation*} \nonumber
U^{3, \epsilon}_{x^*}=\left \{ y \in J_{\mathcal{T}}^{\Sigma}:  y(t) < x^*(t) \; \text{for}  \; t \geq \epsilon \right \}.
\end{equation*}

\noindent
For each of the two sequences of NP-stopping times introduced above, consider
$U_{x^*}$ respectively as:\\
1) $U_{x^*}=U^{1, \epsilon}_{x^*}$. \\
2) $\displaystyle{
U_{x^*}=\left\{
\begin{array}{ll}
U^{1, \epsilon}_{x^*} \cap U^{2, \epsilon}_{x^*} & \text{in case 2a (see proof below)} \\
U^{1, \epsilon}_{x^*} \cap U^{3, \epsilon}_{x^*} & \text{in case 2b (see proof below)}
\end{array}
\right.
}$

\vspace{.1in}
\noindent
In both cases 1) and 2) a sequence $\{x^{(n)}\}$ converging to $x^*$ in the metric $d_{QV}$ will be considered.

\vspace{.1in}
\noindent
1) If the sequence $x^{(n)} \in U^{1, \epsilon}_{x^*}$ converges to $x^*$ in the metric $d_{QV}$ then
$x^{(n)} \to x^*$ uniformly on $[0,T]$.
The fact that the sequence of stopping times $\tau_i(x)=\min(iT/n, T)$, for $i=0,1, \ldots$ is strong locally continuous
is an obvious consequence of the uniform convergence of $\{x^{(n)}\}$ to $x^*$ and the continuity
of trajectory $x^*$.

\vspace{.1in}
\noindent
2) Consider that $M(x^*)=L^*$. There are two possible cases in which 
$M(x^*)=L^*$. \\
Case 2a) $\displaystyle{K_{L^*-1} \leq \sup_{t \in [0,T]} x_t^*< K_{L^*}}$\\
Case 2b) $\displaystyle{x_T=K_{L^*}}$ and $\displaystyle{x_t<K_{L^*}}$ for all $t \in [0,T)$

Suppose that we are in case 2a). Consider the sequence $x^{(n)} \in U^{1, \epsilon}_{x^*} \cap U^{2, \epsilon}_{x^*}$ converging to $x^*$.
Let us first prove iii) from Definition \ref{joint-strong-LC}. As $x^{(n)} \in U^{2, \epsilon}_{x^*}$, it clearly follows that
$$
\displaystyle{\sup_{t \in [0,T]} x_t^{(n)}  \geq \sup_{t \in [0,T]} x_t^* \geq K_{L^*}}.
$$
On the other hand, as $x^{(n)}$ converges uniformly to $x^*$, for $n$ large enough
$\displaystyle{\sup_{t \in [0,T]} x_t^{(n)}< K_{L^*+1}}$ too. Then we conclude that
$$
\displaystyle{K_{L^*} \leq  \sup_{t \in [0,T]} x_t^{(n)}  < K_{L^*+1}}.
$$
Therefore for $n$ large enough $M(x^{(n)})=L^*$ so iii) has been proven.\\

\noindent
Let us prove i) from Definition \ref{joint-strong-LC}. As $x^{(n)} \in U^{2, \epsilon}_{x^*}$ we have that $\tau_i(x^{(n)}) \leq \tau_i(x^*)$ for all $i$.
Now fix $\epsilon>0$, then $x^*(t)< K_i$ if $t\leq \tau_i(x^*)-\epsilon$. As $x^{(n)}$ converges uniformly to $x^*$
we also have that $x^{(n)}(t) < K_i$ if $t\leq \tau_i(x^*)-\epsilon$ for $n$ large enough, which implies that
$\tau_i(x^{(n)})>\tau_i(x^*)-\epsilon$ for $n$ large enough. Then
$$
\tau_i(x^*)-\epsilon < \tau_i(x^{(n)}) \leq \tau_i(x^*).
$$
As $\epsilon$ can be chosen as small as wanted then $\displaystyle{\lim_{n \to \infty} \tau_i(x^{(n)}) =\tau_i(x^*)}$ for all $i$,
thus i) is proven.

\noindent
In order to prove ii) from Definition \ref{joint-strong-LC}, notice that:
$$
\left| x^{(n)}(\tau_i(x^{(n)}))-x^*(\tau_i(x^*)) \right| \leq \left| x^*(\tau_i(x^{(n)}))-x^*(\tau_i(x^*)) \right|
+ \left| x^{(n)}(\tau_i(x^{(n)}))- x^*(\tau_i(x^{(n)})) \right|.
$$
The first term in the previous sum converges to 0 because $x^*$ is continuous and $\tau_i(x^{(n)}) \to \tau_i(x^*)$.
The second term also converges to 0 as consequence of the uniform convergence of $x^{(n)}$  to $x^*$.
Then we can conclude that $x^{(n)}(\tau_i(x^{(n)})) \to x^*(\tau_i(x^*))$ as $n \to \infty$, therefore ii) is proven.\\
Case 2b) follows similarly.
\qed

\end{proof}

\subsection{Arbitrage-Free NP-Portfolios for Heston-Type 
Trajectory Space  $J_{\mathcal{T}}^{\Sigma}(x_0)$}  \label{transferingNoArbitrageToVariableVolClass}

This section introduces a specific volatility class $\Sigma$ leading to  an associated trajectory space $J_{\mathcal{T}}^{\Sigma}$; it also describes a class of NP-portfolios that are NP-arbitrage free on this  trajectory space.
This is achieved by making use of Theorem \ref{main-arbitrage-local-cont} which, in turns, requires the introduction of an appropriate stochastic market model. This model is given by a Heston-type stochastic volatility process which is also used to define  the class of volatility functions $\Sigma$.





\begin{eqnarray}\label{Heston-model}
Z_t&=&z_0 \exp \left( \displaystyle{ \int_0^t (\mu-\sigma_s^2/2)ds+\int_0^t \alpha \sigma_s dB_s^{(1)}+\int_0^t \sqrt{1-\alpha^2} \sigma_s dB_s^{(2)}}\right)\\
\sigma_s^2&=&\bar{V}_s \nonumber \\
\bar{V}_s&=&\frac{1}{h}\int_{s-h}^s V_t dt, h>0 \nonumber \\
dV_s&=&k(\theta-V_s)+\xi\sqrt{V_s}dB_s^{(2)}, V_0=v_0, \nonumber
\end{eqnarray}

\noindent where $B^{(1)}$ and $B^{(2)}$ are independent Brownian motions, $0<\alpha<1$
and $k$, $\theta$, $\xi$ are positive real numbers. 
In order for $\bar{V}_s$ to be defined when $s<h$ we will assume that $V_t=v_0$ for 
$t \in [-h,0]$. 
To guarantee that the Cox-Ingersoll-Ross (CIR) process $V$
remains strictly positive we will also assume that $2k \theta \geq \xi^2$ 
(see \cite{Cox-Ing}).\\
\\
The model described in (\ref{Heston-model}) is very similar to the classical Heston model. The main modification is the regularization of the volatility process $\sigma$, which is usually defined as $\sigma_s^2=V_s$.  
If $h$ is small, $V_s$ and $\bar{V}_s$
will be close, meaning that if the Heston model fits empirical returns data, the regularized model also does. Similar arguments have been used previously in order to establish the practical validity of a model, see for example \cite{cheridito}.\\
\\
Let $S_{\sigma}$ be the topological support of process $\sigma$, i.e. the minimal closed subset $A$ of $C[0,T]$ (equipped with the uniform norm topology) such that $P(\sigma(\omega) \in A)=1$. 
Consider now the set $\Sigma=\left\{ x \in S_{\sigma}: x \text{ has finite variation} \right\}$. It can be easily checked, that almost surely the trajectories of the volatility process $\sigma$ are differentiable therefore have finite variation, which implies that $P(\sigma(\omega) \in \Sigma)=1$.
In particular, $\Sigma$ is non-empty, but also that almost surely the trajectories of the price process $Z$ belong to $J_{\mathcal{T}}^{\Sigma}(z_0)$, therefore Condition $C_0$, in Theorem \ref{main-arbitrage-local-cont},  is satisfied.\\
\\
That the process $\sigma$ satisfies a small ball property on $\Sigma$ with respect to the uniform norm
is consequence of the fact that $\Sigma$ is a subset of $S_{\sigma}$, the topological support of process $\sigma$.
Then a direct application of Proposition \ref{prop:brownian-dep} implies that Condition $C_1$ is also satisfied.\\
\\
By conveniently changing the drift, it can be checked that the Heston type model above is arbitrage free. Now we will transfer the no arbitrage
property from this model to the NP-model $J_{\mathcal{T}}^{\Sigma}(z_0)$.\\
\\
Let $\Phi$ be a NP-admissible portfolio strategy defined on $J_{\mathcal{T}}^{\Sigma}(z_0)$
that is given as described in Theorems \ref{simplePortfoliosAreLC} or \ref{loc-cont-portfolio-under-QV}. The sequence of stopping times
that defines $\Phi$ is considered as in Proposition \ref{pr:stoppingtimes-stoch-vol}. This guarantees that
$\Phi$ is locally $V$-continuous on  $J_{\mathcal{T}}^{\Sigma}(z_0)$ under the metric $d_{QV}$.\\
\\
As the trajectories of the Heston model $Z(\omega)$  belong a.s. to $J_{\mathcal{T}}^{\Sigma}(z_0)$ then
we can consider the isomorphic portfolio $\Phi^Z$ on $Z$, defined a.s. by
\begin{equation}\label{eq:iso-port}
\Phi^Z(t,\omega)=\Phi(t,Z(\omega))
\end{equation}
Portfolio $\Phi^Z$ is admissible on $Z$, therefore $\Phi^Z$ is not an arbitrage for this Heston model.
Directly applying Theorem \ref{main-arbitrage-local-cont} we then conclude that $\Phi$ is not a NP-arbitrage portfolio
on $J_{\mathcal{T}}^{\Sigma}(z_0)$.

\subsection{Implications to  Modified Stochastic Heston Volatility Model}
\label{modifiedHestonModel}

Let us  consider a modified Heston stochastic volatility model similar to (\ref{Heston-model}),
the only difference is the addition of a new stochastic term $Y_t$ as follows.

\begin{equation}\label{modif-Heston-model}
Z^m_t=z_0 \exp \left( \displaystyle{ \int_0^t (\mu-\sigma_s^2/2)ds+\int_0^t \alpha \sigma_s dB_s^{(1)}+\int_0^t \sqrt{1-\alpha^2} \sigma_s dB_s^{(2)} + Y_t }\right)
\end{equation}

The stochastic process $Y$ is assumed to be continuous, with null quadratic variation and independent of $B^{(1)}$ and $B^{(2)}$.
Analogously to the Heston model in (\ref{Heston-model}), it can be proven that the modified process in (\ref{modif-Heston-model}) satisfies the conditions $C_0$ and $C_1$, from Theorem \ref{main-arbitrage-local-cont}, relative to the set of trajectories $J_{\mathcal{T}}^{\Sigma}(z_0)$ and the metric $d_{QV}$.\\
\\
We already argued for the fact that NP-portfolios $\Phi$ as described in Theorems \ref{simplePortfoliosAreLC} or \ref{loc-cont-portfolio-under-QV}
do not constitute NP-arbitrage opportunities. The no arbitrage property will be transferred now to the modified Heston model in (\ref{modif-Heston-model}). Towards this end, consider now the isomorphic portfolio 
$\Phi^Z$ defined almost surely  by (\ref{eq:iso-port}).
As $\Phi$ is not an arbitrage on $J_{\mathcal{T}}^{\Sigma}(z_0)$, Theorem \ref{main-arbitrage-local-cont}
can be applied to conclude that $\Phi^Z$ is not an arbitrage for the modified Heston model.\\
\\
It is worth noticing that the conditions imposed on $Y$ are not very strong, so the model becomes quite flexible.
For example, if $Y=B^H$ is a fractional Brownian motion with $1/2< H \leq 3/4$,
the price process $Z$ will not be a semimartingale.\\

\section{Overview}  \label{overview}

The publication \cite{AFO2011} proposes a trajectory based modeling of financial markets. The main strategy put forward in order to establish no arbitrage results is to connect the proposed trajectory based models with a classical
stochastic reference market model.
This connection is achieved through imposing continuity hypothesis and a density condition in the form of small balls.
The present paper continues and strengthens this line of research by incorporating a richer class of practical portfolios defined through NP-stopping times.  It turns out that realistic trajectory sets and an associated large class of practical portfolios can be defined providing  NP arbitrage free models. This indicates the plausibility of pursuing trajectory based market models. It is natural to expect that many of the results
in the paper can be extended, for example several more examples of sequences of stopping times could  be proven to be jointly strong locally continuous, we have refrained from doing so given the technical demands of the proofs.

In the case of complete markets one can also establish trajectory per trajectory  hedging results and define a natural minmax based pricing methodology that covers the incomplete market case as well (as described in \cite{AFO2011}). Reference \cite{degano}
contains a detailed development of this pricing technique in the discrete case for incomplete market models. Moreover, this last reference establishes a no arbitrage
result, for discrete trajectory based markets, that does not require any reference to a stochastic market model.\\


\appendix\normalsize{{\bf Appendix.} Technical Results and Proofs.} \label{tech-results}

\begin{lemma}\label{Lemma-correspondence-jumps}
Let $0=\tau_0 \leq \tau_1\leq \tau_2\leq \cdots \leq T$  be a jointly strong locally continuous sequence
of NP-stopping times (as per Definition \ref{joint-strong-LC}) defined on $\mathcal{J}_\mathcal{T}^{\sigma, C}(x_0)$ with respect to the Skorohod's metric. Assume that $\inf_{c \in C}|c|>0$. 
Fix $x^* \in \mathcal{J}_\mathcal{T}^{\sigma, C}(x_0)$. Then, there exists an open set $U_{x^*} \subset \mathcal{J}_\mathcal{T}^{\sigma, C}(x_0)$
such that $x^* \in \overline{U}_{x^*}$ and whenever $x_n \to x^*$ in $U_{x^*}$ we have that:

$$ \sum_{s \in (\tau_i(x_n), \tau_{i+1}(x_n)]} 1_{\mathbb{R}\setminus\{0\}}(x_n(s)-x_n(s-))$$ 

converges to 

$$\sum_{s \in (\tau_i(x^*), \tau_{i+1}(x^*)]} 1_{\mathbb{R}\setminus\{0\}}(x^*(s)-x^*(s-))$$

as $n$ approaches infinity for all $i \geq 0$.

\end{lemma}

\begin{proof}
Fix $x^* \in \mathcal{J}_\mathcal{T}^{\sigma, C}(x_0)$. As $0=\tau_0 \leq \tau_1\leq \tau_2\leq \cdots \leq T$
is jointly strong locally continuous, there exists an open set $U_{x^*} \subset \mathcal{J}_\mathcal{T}^{\sigma, C}(x_0)$ 
as in Definition \ref{joint-strong-LC}. Let $\left\{x_n\right\}_{n \geq 1}$ be any sequence  of elements
in  $U_{x^*}$ converging to $x^*$ in the Skorohod's topology. Now we will consider two possible cases:

Case 1: Consider that $x^*$ jumps at the points $y_1< y_2 ...< y_m$ in the open interval 
$\left(\tau_i(x^*), \tau_{i+1}(x^*)\right)$. As $x_n \to x^*$ in the Skorohod's topology,
there exists an increasing function $\lambda_n:[0,T] \to [0,T]$ with $\lambda_n(0)=0$ and
$\lambda_n(T)=T$ such that both $\lambda_n(t)-t \to 0$ and $x_n(\lambda_n(t))-x^*(t) \to 0$ 
uniformly in $[0,T]$. We know from Lemma 2 in \cite{AFO2011} that for $n$ large enough 
the trajectory $x_n$ jumps at the points $\lambda_n(y_1)<\lambda_n(y_2)< ... <\lambda_n(y_m)$.
As the sequence $\left\{ \tau_n \right \}_{n=0,1, ...}$ is jointly strong locally continuous 
we have that $\tau_i(x_n) \to \tau_i(x^*)$ and $\tau_{i+1}(x_n) \to \tau_{i+1}(x^*)$
as $n$ approaches infinity. On the other hand we know that $\lambda_n(t) \to t$ for $t \in [0,T]$
as $n$ approaches infinity. Then we can conclude that
$$\tau_i(x_n)<\lambda_n(y_1)<\lambda_n(y_2)< ... <\lambda_n(y_m) <\tau_{i+1}(x_n)$$
meaning that if $n$ is large enough, for every jump of $x^*$ in the open interval 
$\left(\tau_i(x^*), \tau_{i+1}(x^*)\right)$ 
there is a jump of $x_n$ in $\left(\tau_i(x_n), \tau_{i+1}(x_n)\right)$.

Case 2: Now suppose that  $x^*$ jumps exactly at the point $\tau_{i+1}(x^*)$. We know that for $n$ large enough
$x_n$ will jump at the point $\lambda_n(\tau_{i+1}(x^*))$. The triangle inequality implies that
\begin{equation} \nonumber
|x_n(\tau_{i+1}(x_n))-x_n(\lambda_n(\tau_{i+1}(x^*)))| \leq |x_n(\tau_{i+1}(x_n))-x^*(\tau_{i+1}(x^*))| +
\end{equation}
\begin{equation} \nonumber
|x_n(\lambda_n(\tau_{i+1}(x^*)))-x^*(\tau_{i+1}(x^*))|.
\end{equation}

The first term in the right hand side converges to 0 as n approaches infinity because 
the sequence $\left\{ \tau_n \right \}_{n=0,1, ...}$ is jointly strong locally continuous.
The second term in the right hand side converges to 0 because $x_n \to x^*$ in the Skorohod's topology.
Then we conclude that $|x_n(\tau_{i+1}(x_n))-x_n(\lambda_n(\tau_{i+1}(x^*)))|$ approaches 0 as $n$
approaches infinity, meaning that the point $\tau_{i+1}(x_n)\geq \lambda_n(\tau_{i+1}(x^*))$. 
Then we can conclude that if $x^*$ jumps exactly at the point $\tau_{i+1}(x^*)$, then
for $n$ large enough, the trajectory $x_n$ jumps at $\lambda_n(\tau_{i+1}(x^*))$, and this point 
satisfies that $\tau_{i}(x_n) < \lambda_n(\tau_{i+1}(x^*)) \leq \tau_{i+1}(x_n)$.\\

Cases 1 and 2 imply that 

$$ \sum_{s \in (\tau_i(x_n), \tau_{i+1}(x_n)]} 1_{\mathbb{R}\setminus\{0\}}(x_n(s)-x_n(s-))$$ 

converges to 

$$\sum_{s \in (\tau_i(x^*), \tau_{i+1}(x^*)]} 1_{\mathbb{R}\setminus\{0\}}(x^*(s)-x^*(s-))$$

as $n$ approaches infinity for all $i \geq 0$.
\qed
\end{proof}

\noindent
{\bf Proof of Propositon} \ref{pr:stop-tim-ex}.
\vspace{-.1in}
\begin{proof}
Fix $x^* \in \mathcal{J}_\mathcal{T}^{\sigma, C}(x_0)$, then $x^*(t)=x_0 e^{\sigma z^*(t)}\prod_{i=1}^{n^*(t)} (1+a_i^*)$ for some
$z^* \in \mathcal{Z}_{\mathcal{T}}([0, T])$, $n^*(t)=\sum_i 1_{[0,t]}(s_i^*) \in \mathcal{N}([0,T])$,  and real numbers $a_i^* \in C$, $i=1,2,\ldots,n^*(T)$.
The proof of this proposition strongly relies on finding the appropriate open sets $U_{x^*}$ for each case.  Before constructing
these sets let us introduce some notation. Considering that
any element $y \in \mathcal{J}_\mathcal{T}^{\sigma, C}(x_0)$ has the form $y(t)= e^{\sigma z^y(t)}\prod_{i=1}^{n^y(t)} (1+a_i^y)$
with $z^y \in \mathcal{Z}_{\mathcal{T}}([0, T])$, $n^y(t)=\sum_i 1_{[0,t]}(s_i^y) \in \mathcal{N}([0,T])$,  and $a_i^y \in C$, define:\\

\begin{equation} \nonumber
U^{1, \epsilon}_{x^*}=\left \{ y \in \mathcal{J}_\mathcal{T}^{\sigma, C}(x_0): 0< d_S(y,x^*)< \epsilon \right \},
\end{equation}

\begin{equation} \nonumber
U^{2}_{x^*}=\left \{ y \in \mathcal{J}_\mathcal{T}^{\sigma, C}(x_0):  s_i^y < s_i^*\; \text{for}  \; i \leq n^*(T) \right \},
\end{equation}

\begin{equation} \nonumber
U^{3,\epsilon}_{x^*}=\left \{ y \in \mathcal{J}_\mathcal{T}^{\sigma, C}(x_0): z^*(t)-z^y(t)<0 \; \text{for} \; \epsilon \leq t \leq T \right \},
\end{equation}

\begin{equation}  \nonumber
U^{4}_{x^*}=\left \{ y \in \mathcal{J}_\mathcal{T}^{\sigma, C}(x_0):  a_i^*-a_i^y < 0 \; \text{for} \; i \leq n^*(T) \right \},
\end{equation}

\begin{equation}  \nonumber
U^{5}_{x^*}=\left \{ y \in \mathcal{J}_\mathcal{T}^{\sigma, C}(x_0):  (s_i^y - s_i^*)a_i^*<0 \; \text{for} \; i \leq n^*(T) \right \}.
\end{equation}

\begin{equation} \nonumber
U^{6,\epsilon}_{x^*}=\left \{ y \in \mathcal{J}_\mathcal{T}^{\sigma, C}(x_0): z^*(t)-z^y(t)>0 \; \text{for} \; \epsilon \leq t \leq T \right \},
\end{equation}

\begin{equation}  \nonumber
U^{7}_{x^*}=\left \{ y \in \mathcal{J}_\mathcal{T}^{\sigma, C}(x_0):  a_i^*-a_i^y > 0 \; \text{for} \; i \leq n^*(T) \right \},
\end{equation}

\begin{equation}  \nonumber
U^{8}_{x^*}=\left \{ y \in \mathcal{J}_\mathcal{T}^{\sigma, C}(x_0):  (s_i^y - s_i^*)a_i^*>0 \; \text{for} \; i \leq n^*(T) \right \}.
\end{equation}


\noindent 
For each of the previous sequences of NP-stopping times, consider $U_{x^*}$ respectively as:\\
\\
1) $U_{x^*}=U^{1,\epsilon}_{x^*} \bigcap U^{2}_{x^*}.$\\
2) $U_{x^*}=U^{1,\epsilon}_{x^*}.$\\
3) $U_{x^*}=\left \{
\begin{array}{l l}
U^{1,\epsilon}_{x^*}  \bigcap U^{3,\epsilon}_{x^*} \bigcap U^{4}_{x^*} \bigcap U^{5}_{x^*} & \text{ in cases 3a and 3b (see proof below),}\\
U^{1,\epsilon}_{x^*}  \bigcap U^{6,\epsilon}_{x^*} \bigcap U^{7}_{x^*} \bigcap U^{8}_{x^*} & \text{ in cases 3c and 3d (see proof below).}
\end{array}
\right.$


The fact that $U_{x^*}$ for each of the three cases is an open set is a consequence of Lemma 2 in \cite{AFO2011}.

\vspace{.1in}
\noindent
In  these three cases, a sequence $\{x^{(n)}\}$ converging to $x^*$ will be considered.
As $\{x^{(n)}\}$ converges to $x^*$ in the Skorohod's metric, then there exists a sequence of increasing
functions ${\lambda_n(t)}$ satisfying $\lambda_n(0)=0$, $\lambda_n(T)=T$ such that:

\begin{equation}\label{Skorohod's condition}
\left| x^*(t)- x^{(n)}(\lambda_n(t)) \right| \to 0 \; \mbox{ uniformly on } [0,T] \; as  \; n \to \infty
\end{equation}
and

\begin{equation}\label{Skorohod's condition lambda}
\left| \lambda_n(t)-t \right| \to 0 \; \mbox{ uniformly on } [0,T] \; as  \; n \to \infty.
\end{equation}

\noindent Case 1):\\
\\
Consider a sequence of trajectories $\{x^{(n)}\}$ in $U^{1,\epsilon}_{x^*} \bigcap U^{2}_{x^*}$ converging to $x^*$ in the
Skorohod's topology.\\
\\
Given that the NP-stopping times $\tau_i(x)=\min(iT/N, T)$ do not depend on the trajectory $x$, properties i) and iii)
in Definition \ref{joint-strong-LC} are clearly satisfied.\\
\\
Now let us prove ii). Consider any $s \in [0,T]$. Then
\begin{eqnarray}\label{eq:triangular ineq}
\left|x^{(n)}(s)-x(s)\right|&=&\left|x^{(n)}(s)- x^*\left(\lambda_n^{-1}(s)\right)+ x^*\left(\lambda_n^{-1}(s)\right)-x^*(s) \right|\\
 &\leq& \left|x^{(n)}(s)- x^*\left(\lambda_n^{-1}(s)\right) \right|+\left| x^*\left(\lambda_n^{-1}(s)\right)-x^*(s) \right|. \nonumber
\end{eqnarray}
The term  $\displaystyle{\left|x^{(n)}(s)- x^*\left(\lambda_n^{-1}(s)\right) \right|}$ converges to 0 as n goes to infinity
as consequence of (\ref{Skorohod's condition}).\\
\\
From (\ref{Skorohod's condition lambda}) we get that $\lambda_n^{-1}(s) \to s$ as n goes to infinity, therefore
if $x^*$ is continuous at s, we obtain that $\displaystyle{\left| x^*\left(\lambda_n^{-1}(s)\right)-x^*(s) \right| \to 0}$ as n goes to infinity.\\
\\
If $x^*$ has a jump at s, meaning that $s=s_i^*$ for some i, we have from Lemma 2 in \cite{AFO2011} that there exists an integer number
$N_0$ such that if $n>N_0$, $\lambda_n(s_i^*)=s_i^{x^{(n)}}$. Moreover, as $x^{(n)} \in U^{2}_{x^*}$ we have that $s_i^{x^{(n)}}<s_i^*$,
therefore $\lambda_n(s_i^*)< s_i^*$ and $s_i^*< \lambda_n^{-1}(s_i^*)$. This means that $\lambda_n^{-1}(s_i^*)$ converges to $s_i^*$
from the right. As $x^*$ is right continuous and $s=s_i^*$, we have that
$\displaystyle{\left| x^*\left(\lambda_n^{-1}(s)\right)-x^*(s) \right|} $ also converges to 0 if $x^*$ has a jump at s.\\
\\
From (\ref{eq:triangular ineq}) we conclude that $\displaystyle{\left|x^{(n)}(s)-x(s)\right| \to 0}$ for any $s \in [0,T]$.
In particular this will be true for $s = \tau_i(x)$, thus ii) is proven.\\
\\
Case 2):\\
\noindent From Lemma 2 in \cite{AFO2011} it follows that  for any sequence ${x^{(n)}}$ converging to $x^*$:\\
\\
i) $\lim_{n \rightarrow \infty}  M(x^{(n)})=M(x^*)$.\\
ii) $\lim_{n \rightarrow \infty} \tau_i(x^{(n)}) =\tau_i(x^*)$.\\

\noindent
Again, as a consequence of Lemma 2 in \cite{AFO2011}, there exists an integer number
$N_0$ such that if $n>N_0$, $\lambda_n(s_i^*)=s_i^{x^{(n)}}$. Therefore, if
$n>N_0$ we have:

\[
\left| x^*(s_i^*)- x^{(n)}(\lambda_n(s_i^*)) \right| = \left| x^*(s_i^*)- x^{(n)}(s_i^{x^{(n)}}) \right|.
\]
Using  (\ref{Skorohod's condition}) we have that $\left| x^*(s_i^*)- x^{(n)}(s_i^{x^{(n)}}) \right| \to 0$, so:\\
\\
iii)  $\lim_{n \rightarrow \infty} x_n(\tau_i(x_n)) =x^*(\tau_i(x^*))$.\\
Therefore, the joint strong local continuity property has been proven.\\
\\
Case 3:\\
Consider that $M(x^*)=L^*$. There are four possible cases in which 
$M(x^*)=L^*$:\\

\noindent Case 3a:  
$\displaystyle{K_{L^*-1}< \sup_{t \in [0,T]} x_t^*< K_{L^*}}$\\
\noindent Case 3b: 
$\displaystyle{\sup_{t \in [0,T]} x_t^*= K_{L^*-1}}$ and there exists 
$s \in [0,T)$ such that $x_s^*=K_{L^*-1}$\\
\noindent Case 3c: 
$\displaystyle{\sup_{t \in [0,T]} x_t^*= K_{L^*}}$ and $x_t^*<K_{L^*}$ for all $t \in [0,T]$\\
\noindent Case 3d: 
$\displaystyle{\sup_{t \in [0,T]} x_t^*= K_{L^*}}$ and $x_t^*<K_{L^*}$ for all $t \in [0,T)$, $x_T^*=K_{L^*}$\\

\noindent Consider the cases 3a and 3b at the same time. We have that
$\displaystyle{K_{L^*-1}\leq \sup_{t \in [0,T]} x_t^*< K_{L^*}}$.     
Then, consider a sequence of trajectories $\{x^{(n)}\}$ in $U^{1,\epsilon}_{x^*}  \bigcap U^{3,\epsilon}_{x^*} \bigcap U^{4}_{x^*} \bigcap U^{5}_{x^*}$ converging to $x^*$ in the Skorohod's topology.\\
We will first prove prove iii) in Definition \ref{joint-strong-LC}.

As $\{x^{(n)}\} \to x^*$ in the Skorohod's topology,
it is easy to check that there exists $N_0 \in \mathbb{N}$, depending on $x^*$ such that if $n>N_0$ then
$\displaystyle{\sup_{t \in [0,T]} x_t^{(n)}<K_{L^*}}$.\\
\\
As trajectories $\{x^{(l)}\}$ belong to $U^{1,\epsilon}_{x^*}$, from Lemma 2 in \cite{AFO2011} we know that there exists
$N_1$ such that if $l>N_1$, then $n^{(x^{(l)})}(T)=n^*(T)$, meaning that for $l$ large enough, trajectory $x^{(l)}$ has
exactly the same number of jumps as $x^*$, moreover, the jump times of $x^{(l)}$ are close to the jump times of $x^*$.
Additionally, as trajectories $\{x^{(l)}\}$ belong simultaneously to $U^{4}_{x^*}$ and $U^{5}_{x^*}$, it can be
verified that if $l>N_1$, then for all $t \in [0,T]$:

\begin{equation}\label{eq:comparing-jumps}
\prod_{i=1}^{n^{(x^{(l)})}(t)}\left( 1+a_i^{x^{(l)}} \right) \geq \prod_{i=1}^{n^*(t)} \left(1+a_i^* \right).
\end{equation}
Given that  $\{x^{(l)}\}$ belongs to $U^{3,\epsilon}_{x^*}$, we also have that:
\begin{equation}\label{eq:comparing-brown}
e^{\sigma z^{(x^{(l)})}(t)} >  e^{\sigma z^*(t)}
\end{equation}
for $\epsilon \leq t \leq T$. Combining expressions (\ref{eq:comparing-jumps}) and (\ref{eq:comparing-brown}), we can see that
if $l>N_1$  holds, then:
\begin{equation}\label{eq:comparing-all}
x^{(l)}(t)=x_0 e^{\sigma z^{(x^{(l)})}(t)}\prod_{i=1}^{n^{(x^{(l)})}(t)}\left( 1+a_i^{x^{(l)}} \right)
> x_0 e^{\sigma z^*(t)} \prod_{i=1}^{n^*(t)} \left(1+a_i^* \right) = x^*(t),
\end{equation}
for all $t \in [\epsilon,T]$. Therefore, for $l$ large enough
\[
 K_{L^* } > \sup_{t \in [0,T]} x_t^{(l)} \geq \sup_{t \in [0,T]} x_t^* ,
\]
which implies  that $M(x^{(l)})=L^*$ for $l$ large enough so
\[
\lim_{l \to \infty} M(x^{(l)})=L^*=M(x^*)
\]
and hence iii) in Definition \ref{joint-strong-LC} has been proven.\\
\\

Now let us prove i) in Definition \ref{joint-strong-LC}. From (\ref{eq:comparing-all}) we know that for $l$ large enough
it holds $x^{(l)}(t)>x^*(t)$ for all $t \in [\epsilon,T]$, therefore
$\tau_i(x^{(l)}) \leq \tau_i(x^*)$, for $i=1,2,\ldots,M(x^*)-1$. For 
$i=M(x^*)$, we have that $\tau_i(x^*)=T=\tau_i(x^{(l)})$.
Fix now $\epsilon'>0$, for any $t$ such that $\tau_i(x^*)-\epsilon'<t < \tau_i(x^*)$, the definition of $\tau_i$
implies that $x^*(s)<K_i$ for all $0 \leq s \leq t$.
Then, the convergence of $x^{(l)}$ to $x^*$ in the Skorohod's metric implies that for $l$ large
enough, $x^{(l)}(s)<K_i$ for all $0 \leq s \leq \lambda_l(t)$, meaning that for $l$ large
enough $\lambda_l(t) < \tau_i(x^{(l)})$. On the other hand, as $\lambda_l$ is strictly increasing, we have
$\lambda_l(\tau_i(x^*)-\epsilon')<\lambda_l(t)$. All this implies that for $l$ large enough
\begin{equation}\label{chain-comparison}
\lambda_l(\tau_i(x^*)-\epsilon')<\lambda_l(t) < \tau_i(x^{(l)}) \leq \tau_i(x^*)
\end{equation}
therefore
\[
\lambda_l(\tau_i(x^*)-\epsilon')-\tau_i(x^*)< \tau_i(x^{(l)})- \tau_i(x^*) \leq 0.
\]

\noindent When $l$ approaches infinity the expression in the left hand side approaches $-\epsilon'$.
As $\epsilon'$ can be chosen as small as we want, then the Squeeze Theorem implies
that $\tau_i(x^{(l)}) \to \tau_i(x^*)$ as $l$ approaches infinity, thus i) is proven.\\
\\
In order to prove ii) in Definition \ref{joint-strong-LC}, notice that the triangle inequality gives
\begin{eqnarray} \label{triangleInequality}
\left|x^*(\tau_i(x^*))-x^{(l)}(\tau_i(x^{(l)})) \right| &\leq & \left|x^*(\tau_i(x^*))-x^*(\lambda_l^{-1}(\tau_i(x^{(l)}))) \right| \nonumber \\
                                                  && + \left|x^*(\lambda_l^{-1}(\tau_i(x^{(l)})))- x^{(l)}(\tau_i(x^{(l)}))\right|.
\end{eqnarray}

\noindent
As a consequence of (\ref{chain-comparison}) we obtain
\[
\tau_i(x^*)-\epsilon < \lambda_l^{-1} (\tau_i(x^{(l)})).
\]
As $\epsilon$ can be chosen as small as wanted, it follows that $\lambda_l^{-1} (\tau_i(x^{(l)}))$ approaches  $\tau_i(x^*)$ from the right
as $l$ approaches infinity. Then, the right continuity of $x^*$ implies that
\[
\left|x^*(\tau_i(x^*))-x^*(\lambda_l^{-1}(\tau_i(x^{(l)}))) \right| \to 0 \; \textrm{as} \; l \to \infty.
\]
On the other hand
\[
\left|x^*(\lambda_l^{-1}(\tau_i(x^{(l)})))- x^{(l)}(\tau_i(x^{(l)}))\right| \to 0 \; \textrm{as} \; l \to \infty
\]
as a consequence of the convergence of $x^{(l)}$ to $x^*$ in the Skorohod's metric.\\

\noindent As both terms in the right hand side of (\ref{triangleInequality}) converge to 0,
then the left hand side also converges to 0, so ii) is proven.\\

If trajectory $x^*$ falls in one of the cases 3c or 3d, the joint strong locally continuity property can be proved similarly. The main difference in the proof is that a sequence $\{x^{(n)}\}$ belonging to 
$U^{1,\epsilon}_{x^*}  \bigcap U^{6,\epsilon}_{x^*} \bigcap U^{7}_{x^*} \bigcap U^{8}_{x^*}$ and converging to $x^*$ will satisfy that 
$x^{(l)}(t)<x^*(t)$ if $t \in (\epsilon, T].$

\qed
\end{proof}

\begin{proposition} \label{positiveInOpenSet}
Let $f: \mathcal{X} \rightarrow \mathbb{R}$ be a locally continuous function.
Consider $x^* \in \mathcal{X}$ and an arbitrary open interval $I$ such that  $f(x^*) \in I$. Then, there exists an open set $V_{x^*} \subset \mathcal{X}$, with $x^{\ast} \in \overline{V}_{x^*}$,
such that $f(x) \in I $ for all $x \in V_{x^*}$.
\end{proposition}

The proof of Proposition \ref{positiveInOpenSet} is trivial so we are not including it here.

\end{document}